\documentclass[oneside]{amsart}
\usepackage[T1]{fontenc}
\usepackage[latin9]{inputenc}
\usepackage{amssymb}
\usepackage{amsmath}
\usepackage{pstricks, pst-node, pst-tree}
\usepackage{rotating}
\usepackage{caption}
\usepackage{float}

\newcommand{\sbs}[1]{_{(#1)}}
\newcommand{\charf }{\mathrm{char}\,\fie }

\newcommand{\BG}[1]{\mathbb{B}_{#1}}
\newcommand{\C}{\mathbb{C}}
\newcommand{\N}{\mathbb{N}}
\newcommand{\F}{\mathbb{F}}
\newcommand{\cO}{\mathcal{O}}
\newcommand{\dimker}{\dim\ker(1+c_{12}+c_{12}c_{23})}
\newcommand{\End }{\mathrm{End}}
\newcommand{\fie}{\Bbbk}
\newcommand{\Hilb}{\mathcal{H}}
\newcommand{\id}{\mathrm{id}}
\newcommand{\imm}[1]{\mathrm{imm}_{#1}}
\newcommand{\Inn}{\mathrm{Inn}}
\newcommand{\lcoa}{\delta }
\newcommand{\ndN }{\mathbb{N}}
\newcommand{\ndQ }{\mathbb{Q}}
\newcommand{\ndZ }{\mathbb{Z}}
\newcommand{\nro}[1]{l^{(3)}_{#1}}
\newcommand{\reG }[1]{\overline{G}_{#1}}
\newcommand{\rker}{\ker(1+c_{12}+c_{12}c_{23})|_{V_{\cO}^{\otimes3}}}
\newcommand{\rdimker}{\dim\rker }
\newcommand{\rk}{\mathrm{rank}\,}
\newcommand{\SG}[1]{\mathbb{S}_{#1}}
\newcommand{\supp }{\mathrm{supp}\,}
\newcommand{\toba}{\mathfrak{B}}
\newcommand{\trid}{\triangleright}
\newcommand{\tup}{\bar{x}}
\newcommand{\ydG}{ {}_{\fie G}^{\fie G}\mathcal{YD}}

\newcommand{\ttriangle}[6]{
\ncline[linestyle=dotted]{->}{#1,#2}{#3,#4}
\ncline[linestyle=dotted]{->}{#3,#4}{#5,#6}
\ncline[linestyle=dotted]{->}{#5,#6}{#1,#2}
}

\newcommand{\striangle}[6]{
\ncline{->}{#1,#2}{#3,#4}
\ncline{->}{#3,#4}{#5,#6}
\ncline{->}{#5,#6}{#1,#2}
}

\theoremstyle{plain}
\newtheorem{theorem}{Theorem}
\newtheorem{lemma}[theorem]{Lemma}
\newtheorem{prop}[theorem]{Proposition}
\newtheorem{corollary}[theorem]{Corollary}
\theoremstyle{definition}
\newtheorem{definition}{Definition}
\newtheorem{example}{Example}
\theoremstyle{remark}
\newtheorem{remark}{Remark}

\begin{document}

\title[Braided racks, Hurwitz actions, Nichols algebras]{Braided racks,
Hurwitz actions and Nichols algebras with many cubic relations}

\author{I. Heckenberger}
\address{ 
Philipps-Universit\"at Marburg\\
FB Mathematik und Informatik \\ 
Hans-Meerwein-Stra\ss e\\
35032 Marburg, Germany}
\email{heckenberger@mathematik.uni-marburg.de}
\author{
A. Lochmann}
\address{ 
Philipps-Universit\"at Marburg\\ 
FB Mathematik und Informatik \\
Hans-Meerwein-Stra\ss e\\
35032 Marburg, Germany}
\email{lochmann@mathematik.uni-marburg.de}
\author{
L. Vendramin}
\address{
Depto. de Matem\'atica, FCEyN\\
Universidad de Buenos Aires\\
Pab. 1, Ciudad Universitaria (1428)\\
Buenos Aires, Argentina}
\email{lvendramin@dm.uba.ar}


\setcounter{tocdepth}{1}

%
%



\maketitle

\begin{abstract}
We classify Nichols algebras of irreducible Yetter-Drinfeld modules over groups
such that the underlying rack is braided and the homogeneous component of
degree three of the Nichols algebra satisfies a given inequality. This
assumption turns out to be equivalent to a factorization assumption on the
Hilbert series.  Besides the known Nichols algebras we obtain a new example.
Our method is based on a combinatorial invariant of the Hurwitz orbits with
respect to the action of the braid group on three strands. 
\end{abstract}

\section*{Introduction}

Since its introduction in 1998 by Andruskiewitsch and Schneider, the Lifting
Method \cite{MR1659895} grew to one of the most powerful and most fruitful
methods to study Hopf algebras \cite{MR1780094}, \cite{MR1952612},
\cite{MR2176141}, \cite{MR2504492}, \cite{MR2651563}, \cite{MR2678630},
\cite{MR2732981}, \cite{MR2630042}, \cite{MPSW}, \cite{GG}, \cite{M}. Although
it originates from a purely Hopf algebraic problem, the method quickly showed a
strong relationship with other areas of mathematics such as
\begin{itemize}
  \item quantum groups \cite{MR1632802}, \cite{MR2630042},
  \item noncommutative differential geometry \cite{MR994499}, \cite{MR1396857},
	  \cite{MR2106930}, \cite{MR1492989},
  \item knot theory \cite{MR1470954}, \cite{MR1990571}, \cite{MR1918807},
  \item combinatorics of root systems and Weyl groups \cite{MR2462836},
	  \cite{AHS}, \cite{Ang},
  \item Lyndon words \cite{MR1763385}, \cite{MR2331765}, \cite{MR2491909},
  \item cohomology of flag varieties \cite{MR1667680}, \cite{MR2209265},
	  \cite{MR2558748},
  \item projective representations \cite{twisting},
  \item conformal field theory \cite{MR2030633}, \cite{ST}.
\end{itemize}
The heart of the Lifting Method is formed by
the structure theory of Nichols algebras.
Nichols algebras have been studied first by Nichols \cite{MR0506406}.
These are connected graded braided Hopf algebras \cite{MR1907185}
generated by primitive elements, and
all primitive elements are of degree one.
If the braiding is trivial and the
base field has characteristic $0$, the Nichols algebra is a polynomial ring.
The situation becomes much more complicated for non-trivial braidings. A major
problem, which is open since the introduction of the Lifting Method, is the
classification of finite-dimensional Nichols algebras over groups
\cite[Questions\,5.53, 5.57]{MR1907185}.
Under the additional assumption that the base field has characteristic $0$
and the group is abelian,
this problem was completely solved in \cite{MR2207786,MR2462836} using Lie theoretic structures.
A generalization of this theory to arbitrary groups is possible
\cite{AHS,MR2734956} and opens new research directions \cite{HS}. Nevertheless,
the problem of classifying finite-dimensional Nichols algebras of irreducible
Yetter-Drinfeld modules over non-abelian groups cannot be attacked with this method.
One needs a fundamentally new idea. One approach in this direction is to
identify finite groups admitting (almost) only infinite-dimensional Nichols algebras.
Here a remarkable progress could be achieved
for sporadic simple groups and for alternating groups \cite{AMPA,sporadic}.
Despite these developments, the structure of important examples of Nichols algebras,
for example those associated with the transpositions of the symmetric groups,
remained unknown since more than 10 years \cite{MR1667680}, \cite{MR1800714}, \cite{AMPA}.

So far only a few finite-dimensional Nichols algebras of irreducible Yetter-Drinfeld modules
over non-abelian groups are known. These examples have an interesting
property in common: the Hilbert series of the Nichols algebras factorize into
the product of polynomials of the form $1+t^r+t^{2r}+\dots +t^{nr}$ with $r,n\ge1 $.
A theoretical explanation of this fact is not known.
Motivated by this observation, in \cite{Marburg} M.~Gra\~na and the first and the
last authors classified
finite-dimensional Nichols algebras over groups
with many quadratic relations. This
corresponds to a factorization of the Hilbert series, where only $r=1$
appears.
After the publication of the paper some other examples appeared which require
to allow $r>1$. In our paper we attack the case $r\le 2$. We consider in
detail the Hurwitz orbits with respect to the action of the braid group
$\BG 3$ on $X^3$, where $X$ is the support of the Yetter-Drinfeld module.
For such orbits,
we obtain an estimate on the kernel of the shuffle map
using graph theoretical structures closely related to those in percolation
theory \cite{MR2594903}, \cite{MR2650054}.
Such structures are known to be very complicated.
Since we are forced to perform very sensitive calculations, we
concentrate on braided racks, see Definition~\ref{def:braided}.
We obtain all known examples of
finite-dimensional Nichols algebras of irreducible Yetter-Drinfeld modules
over non-abelian groups except those
over the affine racks with $5$ elements (which are not braided), and we also get two new examples.
In principal, our method allows us to
consider arbitrary racks, but to do so we will need additional improvements of
the general theory.

Our approach has the advantage that it works for all groups and it produces
quickly all known examples. Surprisingly, during our calculations we never
met any examples of Nichols algebras
which satisfy our assumption but are not known to be
finite-dimensional. Although there exist many indecomposable braided racks,
for example conjugacy classes of $3$-transpositions,
we do not use difficult classification results
such as the classification of $3$-transposition groups
\cite{MR0294487}, or \cite{MR0311765}.

The structure of our paper is as follows. In Section~\ref{sec:BG} we recall
the fundamental notions related to racks with particular emphasis on braided
racks, see Definition~\ref{def:braided}. We recall the Hurwitz action of the
braid group. The orbits of this action play a fundamental role in our
approach. In Proposition~\ref{pro:sizes} we determine the Hurwitz orbits in
$X^3$ for braided racks $X$. The structure of Hurwitz orbits is in general not
known. This is one of the reasons we study braided racks first. In
Section~\ref{sec:BG} we also define and determine the immunity of the Hurwitz
orbits. This will be a crucial ingredient for our classification theorem.

In Section~\ref{sec:NA} we formulate our main theorem concerning Nichols
algebras with many cubic relations. With Propositions~\ref{pro:orbit1} and
\ref{pro:orbit8} we give detailed information on the kernel of the quantum
shuffle map restricted to orbits of size $1$ and $8$. This information will
help us to obtain a condition in Proposition~\ref{pro:gen_ineq}
allowing us to concentrate on a few braided racks. These racks are
classified in Sections \ref{sec:deg2}, \ref{sec:deg4} and \ref{sec:deg36}.
In Section~\ref{sec:deg2} we also mention and use an interesting connection to
$3$-transposition groups \cite{MR0294487}, \cite{MR1423599}.
In Section~\ref{sec:proof} we collect the information obtained in the previous
sections to prove our main theorem. We consider the remaining racks and the
corresponding Nichols algebras case by case. Our careful preparations allow us
to succeed with the proof without using any technical assumptions.

In two appendices we collect tables with information on the racks and
the Nichols algebras found and we display Hurwitz orbits graphically.

\section{Braid groups, racks and Hurwitz actions}
\label{sec:BG}

\subsection{Racks}

We recall basic notions and facts about racks. For additional information we refer to \cite{MR1994219}.
A \textit{rack} is a pair $(X,\trid)$, where $X$ is a non-empty
set and $\trid:X\times X\to X$ is a map (considered as a binary operation on $X$) such that
\begin{enumerate}
  \item the map $\varphi_i:X\to X$, where $x\mapsto i\trid x$, is bijective for all $i\in X$, and
  \item $i\trid(j\trid k)=(i\trid j)\trid(i\trid k)$ for all $i,j,k\in X$.
\end{enumerate}
For all $n\in\N$ and $i,j\in X$ we write $i\trid^n j=\varphi_{i}^n(j)$.

A rack $(X,\trid )$, or shortly $X$, is a \textit{quandle} if $i\trid i=i$ for all $i\in X$.
A \textit{subrack} of a rack $X$ is a non-empty subset $Y\subseteq X$ such that
$(Y,\trid)$ is also a rack.  
The \textit{inner group} of a rack $X$ is the group generated by the
permutations $\varphi_i$ of $X$, where $i\in X$. We write $\Inn(X)$ for the
inner group of $X$.  A rack is said to be \emph{faithful} if the map
\begin{align}
\varphi :X\to\Inn (X),\qquad
i\mapsto\varphi_i,
\label{eq:varphi}
\end{align}
is injective.

\begin{remark}
Let $X$ be a rack. Then 
\begin{gather}
\label{eq:conjugation}
\varphi_{i\trid j}=\varphi_i\varphi_j\varphi_i^{-1}
\end{gather}
for all $i,j\in X$.
\end{remark}

We say that a rack $X$ is \textit{indecomposable} if the inner group $\Inn(X)$
acts transitively on $X$. Also, $X$ is \textit{decomposable} if it is not
indecomposable. Any finite rack $X$ is the disjoint union of indecomposable
subracks \cite[Prop.\,1.17]{MR1994219} called the \textit{components of} $X$.

Let $(X,\trid)$ and $(Y,\trid)$ be racks.
A map $f:X\to Y$ is
a \textit{morphism} of racks if $f(i\trid j)=f(i)\trid f(j)$ for all $i,j\in
X$.

\begin{example}
\label{exa:racks}
A group $G$ is a rack with $x\trid y=xyx^{-1}$ for all $x,y\in G$. 
If a subset $X\subseteq G$ is stable under conjugation by $G$, then it is a
subrack of $G$.
In particular, we list the following examples.
\begin{enumerate}
\item The rack given by the conjugacy class of involutions in $G=\mathbb{D}_p$, the dihedral
group with $2p$ elements, has $p$ elements. It is called the \textit{dihedral rack} (of order $p$) 
and will be denoted by $\mathbb{D}_p$.
\item The rack $\mathcal{T}$ is the rack associated to the conjugacy class of $(2\,3\,4)$
in $\mathbb{A}_4$. This is the rack associated with the vertices of the
tetrahedron, see \cite[\S 1.3.4]{MR1994219}.
\item The rack $\mathcal{A}$ is the rack associated to the conjugacy class of $(1\,2)$ in $\SG 4$.
\item The rack $\mathcal{B}$ is the rack associated to the conjugacy class of
  $(1\,2\,3\,4)$ in $\SG 4$.
\item The rack $\mathcal{C}$ is the rack associated to the conjugacy class of
  $(1\,2)$ in $\SG 5$.
\end{enumerate}
\end{example}

\begin{example}
The racks $\mathbb{D}_p$ ($p$ a prime number), $\mathcal{T}$, $\mathcal{A}$, $\mathcal{B}$,
$\mathcal{C}$ are faithful and indecomposable.
\end{example}

\begin{example}
\label{exa:affine}
Let $A$ be an abelian group and let $X=A$. For any $g\in\mathrm{Aut}(A)$ we have a
rack structure on $X$ given by 
\[
x\trid y=(1-g)x+gy
\]
for all $x,y\in X$.  This rack is called the \textit{affine rack} associated to
the pair $(A,g)$ and will be denoted by $\mathrm{Aff}(A,g)$. In particular, let
$p$ be a prime number, $q$ a power of $p$ and $\alpha\in\F_q\setminus\{0\}$.
We write $\mathrm{Aff}(\mathbb{F}_q, \alpha)$, or simply
$\mathrm{Aff}(q,\alpha)$, for the affine rack $\mathrm{Aff}(A,g)$, where
$A=\mathbb{F}_q$ and $g$ is the automorphism given by $x\mapsto\alpha x$ for
all $x\in\F_q$.
\end{example}

\begin{example}
A finite affine rack $(A,g)$ is faithful if and only if it is indecomposable,
see \cite[\S 1.3.8]{MR1994219}.
\end{example}

\begin{remark}
\label{rem:inner}
Let $X$ be a finite rack and assume that
$\Inn(X)$ acts transitively on $X$. Then for all $i,j\in
X$ there exist $r\in\mathbb{N}$ and
$k_1,k_2,\dots ,k_r\in X$ such that
$\varphi^{\pm1}_{k_1}\varphi^{\pm1}_{k_2}\cdots\varphi^{\pm1}_{k_r}(i)=j$.
Equation~\eqref{eq:conjugation} implies that
all permutations $\varphi_i$, where $i\in X$, have the
same cycle structure.
\end{remark}

\begin{lemma} \cite[Lemma\,1.14]{MR1994219}
\label{lem:rack_indecomposable}
Let $X$ be a rack, and let
$Y$ be a non-empty proper finite subset of $X$. The following are equivalent.
\begin{enumerate}
\item $Y$ and $X\setminus Y$ are subracks of $X$.
\item $X\trid Y\subseteq Y$. 
\end{enumerate}
\end{lemma}

By \cite[Lemma 2.18]{Marburg} it is possible to define the degree of a finite
indecomposable rack.

\begin{definition}
The \emph{degree} of a finite indecomposable rack $X$ is the number
$\mathrm{ord}(\varphi_{x})$ for some (equivalently, all) $x\in X$.
\end{definition}

For any rack $X$ let $G_X$ denote its enveloping group
\begin{align}
 G_X=&\langle X\rangle /(xy=(x\trid y)x\text{ for all $x,y\in X$}).
\label{eq:reG}
\end{align}
For a finite indecomposable rack $X$ of degree $n$,
the \emph{finite enveloping group} of $X$ 
is defined as $\overline{G_{X}}=G_{X}/\langle x^{n}\rangle$, where $x\in X$.
This definition does not depend on the choice of $x\in X$,
see \cite[Lemma 2.18]{Marburg}.

%

\subsection{Braided racks}
 
\begin{definition}
\label{def:braided}
A rack $X$ is \emph{braided} if $X$ is a quandle and for all $x,y\in X$
at least one of the equations $x\trid(y\trid x)=y$, $x\trid y=y$ holds.
\end{definition}

\begin{lemma}
\label{lem:braided_cycle}
Let $X$ be a braided rack and let $x,y,z\in X$ such that $x\trid y=z$ and $z\ne
y$. Then $y\trid z=x$ and $z\trid x=y$. 
\end{lemma}

\begin{proof}
This follows from Definition \ref{def:braided}.
\end{proof}

\begin{lemma}
\label{lem:crossed}
Let $X$ be a braided rack and let $x,y\in X$. 
\begin{enumerate}
	\item If $y\trid x=x$ then $x\trid y=y$.
	\item If $x\trid (y\trid x)=y$ then $y\trid (x\trid y)=x$.
\end{enumerate}
\end{lemma}

\begin{proof}
	Assume that $x\trid (y\trid x)=y$ and $y\trid x=x$. Then
	$y=x\trid(y\trid x)=x\trid x=x$ and hence $x\trid y=y$, $y\trid(x\trid
	y)=x$. 
\end{proof}

\begin{lemma}
	Let $X$ be a quandle. The following are equivalent.
	\begin{enumerate}
		\item $X$ is braided.
		\item $x\trid(y\trid x)\in\{x,y\}$ for all $x,y\in X$.
	\end{enumerate}
\end{lemma}

\begin{proof}
(1)$\Rightarrow$(2). If $x,y\in X$ with $x\trid (y\trid x)\ne y$ then $x\trid
y=y$. Hence $y\trid x=x$ by Lemma \ref{lem:crossed}. Thus $x\trid(y\trid
x)=x\trid x=x$. 

(2)$\Rightarrow$(1). Let $x,y\in X$. Then $x\trid(y\trid x)\in\{x,y\}$ and
$y\trid(x\trid y)\in\{x,y\}$.  We have to show that $x\trid(y\trid x)=y$ or
$x\trid y=y$. Assume that $x\trid(y\trid x)\ne y$. Then $x\trid (y\trid x)=x$
and hence $y\trid x=x$ since $X$ is a quandle. If $y\trid (x\trid y)=y$ then
$x\trid y=y$. If $y\trid(x\trid y)=x$ then $x=(y\trid x)\trid(y\trid y)=x\trid
y$ and hence $x=y$. Again it follows that $x\trid y=y$.
\end{proof}

\begin{lemma}
\label{lem:braided_is_faithful}
Let $X$ be an indecomposable braided rack. Then $X$ is faithful.
\end{lemma}

\begin{proof}
Assume first that there exists $x\in X$ such that $z\trid x=x$ for all $z\in X$.
Since $X$ is indecomposable, Lemma~\ref{lem:rack_indecomposable} with $Y=\{x\}$
implies that $X=\{x\}$. Then $X$ is faithful.

Let now $x,y\in X$ such that $x\trid z=y\trid z$ for all $z\in X$.
By the previous paragraph we may assume that there exists $z\in X$
such that $z\trid x\ne x$. Then $x=z\trid(x\trid z)=z\trid(y\trid
z)\in\{y,z\}$ and hence $x=y$. Thus $X$ is faithful.
\end{proof}

Let $X$ be a finite indecomposable faithful rack
and let $x\in X$. In \cite[Sect.\,2.3]{Marburg}
integers $k_n$
for $n\in\N_{\geq2}$ were defined by 
\begin{align*}
k_n=\#\{y\in X\mid & \underbrace{x\trid(y\trid(x\trid(y\trid\cdots)))}_{n\text{
elements}}=y,\\
&\underbrace{x\trid(y\trid(x\trid(y\trid\cdots)))}_{j\text{
elements}}\not=y \text{ for all $j\in \{1,2,\dots ,n-1\}$}\}.
\end{align*}
In particular, 
\[
k_2=\#\{y\in X\mid x\trid y=y,x\not=y\},\quad k_3=\#\{y\in X\mid x\trid(y\trid x)=y,x\trid y\not=y\}.
\]
Since $X$ is indecomposable, the integers $k_n$ do not depend on the choice of
$x$. 
\begin{remark}
By definition, an indecomposable rack $X$ is braided if and only if
$X$ is faithful and $k_{n}=0$ for all $n>3$.
\end{remark}

\begin{example}
The racks $\mathbb{D}_{3}$, $\mathcal{T},$ $\mathcal{A},$ $\mathcal{B},$
$\mathcal{C}$ are braided, see \cite[Table 2]{Marburg}.
\end{example}

%

\begin{example}
Let $A$ be a finite abelian group and $g\in\mathrm{Aut}(A)$. It is well-known
that the affine rack $\mathrm{Aff}(A,g)$ is faithful if and only if $1-g$ is
injective.  Since $A$ is finite, this is equivalent to $x\trid y\ne y$ for all
$x,y\in X$ with $x\ne y$. Therefore $\mathrm{Aff}(A,g)$ is braided if and only
if $1-g+g^{2}=0$. In particular, the affine racks
$\mathrm{Aff}(\mathbb{F}_{5},2)$ and $\mathrm{Aff}(\mathbb{F}_{5},3)$ are not
braided, but $\mathrm{Aff}(\mathbb{F}_{7},3)$ and
$\mathrm{Aff}(\mathbb{F}_{7},5)$ are braided. If an affine rack
$\mathrm{Aff}(\mathbb{F}_{q},\alpha)$ is braided, then $\alpha$ has order
$2,\,3$ or $6$. If $\mathrm{ord}(\alpha)=2$, then $q$ is a power of $3$. If
$\mathrm{ord}(\alpha)=3$, then $q$ is a power of $2$. 
\end{example}

\begin{prop}
\label{pro:degree}
Let $X$ be a braided indecomposable rack. Then $X$ has degree $1$, $2$, $3$,
$4$ or $6$.
\end{prop}

\begin{proof}
Let $x,y\in X$ such that $x\trid y\ne y$. Assume that $x\trid^n y=y$ with
$n>4$ minimal. We will prove that $n=6$. We have 
\[
(x\trid y)\trid(x\trid^{2}y)=x\trid(y\trid(x\trid y))=x\trid x=x.
\]
By applying $\varphi_{y}$ we obtain that 
\[
x\trid(y\trid(x\trid^{2}y))=(y\trid(x\trid y))\trid(y\trid(x\trid^{2}y))=y\trid x=x\trid^{n-1}y.
\]
Then
$y\trid(x\trid^{2}y)=x\trid^{n-2}y$. By applying $\varphi_{x\trid^{2}y}$ to the
equation $x\trid(x\trid^{3}y)=x\trid^{4}y$ we obtain that 
$(x\trid^{2}y)\trid(x\trid^{4}y)=y$,
since
\begin{align*}
(x\trid^{2}y)\trid(x\trid(x\trid^{3}y))
& =((x\trid^{2}y)\trid x)\trid((x\trid^{2}y)\trid(x\trid^{3}y))\\
& =(x\trid y)\trid(x\trid^{2}(y\trid(x\trid y)))=(x\trid y)\trid x=y.
\end{align*}
Since $x\trid^{4}y\ne y$, we conclude that $(x\trid^{2}y)\trid(x\trid^{4}y)\ne
x\trid^{4}y$.  Then
\[
((x\trid^{2}y)\trid(x\trid^{4}y))\trid(x\trid^{2}y)=x\trid^{4}y,
\]
because $X$
is braided. Therefore $x\trid^{4}y=y\trid(x\trid^{2}y)=x\trid^{n-2}y$ and hence
the claim holds.
\end{proof}

\begin{prop}
There exist infinitely many finite braided indecomposable racks of degree $6$ which
are generated by two elements.
\end{prop}

\begin{proof}
The affine racks $X=\mathrm{Aff}(\mathbb{F}_{q},\alpha)$ are braided
if and only if $1-\alpha+\alpha^{2}=0$. Take any prime number $p>3$.
If there exists $\alpha\in\mathbb{F}_{p}$ such that $1-\alpha+\alpha^{2}=0$,
then $X=\mathrm{Aff}(\mathbb{F}_{q},\alpha)$ is braided. Otherwise,
take the quadratic extension of $\mathbb{F}_{p}$ by $\alpha$, where
$1-\alpha+\alpha^{2}=0$. These racks are indecomposable, since $\alpha\ne1$.
Moreover, $1-\alpha+\alpha^{2}=0$ implies that $\alpha^{6}=1$. Since
$p>3$, $\alpha^{2}\ne1$ and $\alpha^{3}\ne1$. We claim that these
affine racks are always generated by two elements. If there exists
$\alpha\in\mathbb{F}_{p}$ such that $1-\alpha+\alpha^{2}=0$, the
claim follows from \cite[Prop.\,4.2]{CLA}. Otherwise take the
quadratic extension of $\mathbb{F}_{p}$ by $\alpha$. Then 
\begin{equation}
(u+\alpha v)\trid(x+\alpha y)=(u+v-y)+\alpha(x+y-u)
\label{eq:affinerack}
\end{equation}
for all $u,v,x,y\in\mathbb{F}_{p}$. In particular, $u\trid^{3}0=2u$
for all $u\in\mathbb{F}_{p}$ and hence $\mathbb{F}_{p}$ is included
in $S$, the subrack generated by $0$ and $1$. Since $0\trid1=\alpha$
and $(\alpha v)\trid^{3}0=\alpha(2v)$ for all $v\in\mathbb{F}_{p}$,
we conclude similarly that $\alpha\mathbb{F}_{p}$ is also included in $S$. Therefore
the claim follows from Equation \eqref{eq:affinerack} by taking $(u,v)=(0,k)$
for $k\in\mathbb{F}_{p}$ and $(x,y)=(l,0)$ for $l\in\mathbb{F}_{p}$.
\end{proof}

\subsection{Hurwitz actions}

For any $n\in \ndN $ let
\begin{equation}
\label{eq:braidgroup}
\begin{aligned}
\BG n=\langle \sigma _1,\dots ,\sigma _{n-1}\rangle /
(&\sigma _i\sigma _j=\sigma _j\sigma _i \text{ if $|i-j|\ge 2$, }\\
 &\sigma _i\sigma _j\sigma _i=\sigma _j\sigma _i\sigma _j \text{ if $|i-j|=1$})
\end{aligned}
\end{equation}
denote the braid group on $n$ strands.
According to \cite{MR975077}, the action of $\BG n$ on the set
$X^{n}=X\times\cdots\times X$ ($n$-times), where $X$ is a conjugacy class of a group,
was studied implicitly in \cite{MR1510692}. 

Let $X$ be a rack and let $n\in\ndN $.
There is a unique action of the braid group $\BG n$ on $X^n$ such that
\begin{align}
\sigma _i(x_1,\dots ,x_n)=(x_1,\dots ,x_{i-1},x_i\trid x_{i+1},x_i,x_{i+2},\dots,x_n)
\end{align}
for all $x_1,\dots,x_n\in X$, $i\in \{1,2,\dots,n-1\}$.
This action of $\BG n$ on $X^n$ is called the \emph{Hurwitz action} on $X^n$.
For any $(x_1,x_2,\dots,x_n)\in X^n$
we write $\cO(x_1,x_2,\dots,x_n)$ for its \emph{Hurwitz orbit},
the orbit under the Hurwitz action.
The rack $X$ acts on itself via the map $\trid$.
This extends to a canonical action of
the enveloping group $G_X$ on $X$. More generally,
$G_X$ acts on $X^n$ diagonally:
\begin{align}
  g\trid (x_1,\dots,x_n)=(g\trid x_1,\dots ,g\trid x_n)\quad
  \text{for all $g\in G_X$, $x_1,\dots,x_n\in X$.}
  \label{eq:GXonXn}
\end{align}
The diagonal action of $G_X$ and the action of $\BG n$ on $X^n$ commute.
Two Hurwitz orbits $\cO _1,\cO _2\subseteq X^n$ are called
\textit{conjugate} if there exists $g\in G_X$ such that the map
$X^n\to X^n$, $\tup \mapsto g\trid \tup $,
induces a bijection $\cO _1\to \cO _2$.
Two Hurwitz orbits $\cO _1,\cO _2\subseteq X^n$ are called
\textit{isomorphic} if there exists a bijection $\varphi :\cO _1\to \cO _2$
such that $\varphi (\sigma (\tup))=\sigma(\varphi (\tup))$ for all
$\sigma \in \BG n$, $\tup\in \cO _1$. Clearly, conjugate Hurwitz orbits are
isomorphic.

\begin{remark}
The braided action studied in \cite{Marburg} is the same as the Hurwitz action on $X^2$.
\end{remark}

\begin{remark}
Let $X$ be a rack, $n\in \ndN $ and $x_1,\dots,x_n,y_1,\dots ,y_n\in X$.
By the definition of the enveloping group $G_X$,
if $(y_{1},y_{2},\dots ,y_{n})\in\cO(x_{1},x_{2},\dots ,x_{n})$
then $y_{1}y_{2}\cdots y_{n}=x_{1}x_{2}\cdots x_{n}$ in $G_X$.
\end{remark}

In this work we focus on orbits of the Hurwitz action of $\BG 3$.
For a given rack $X$ and for all $j\in\ndN $ let 
\[
\nro{j}=\#\{\cO(x,y,z)\;|\;x,y,z\in X,\;\#\cO(x,y,z)=j\}.
\]
It should always be clear from the context which rack $X$ is. 

In Proposition~\ref{pro:sizes} below
we determine the Hurwitz orbits $\cO \subseteq
X^3$ of a braided rack $X$ up to isomorphism. The non-trivial orbits are
illustrated in Figures~\ref{fig:H3}--\ref{fig:H24} in Appendix B.
In these figures, circles stay for triples in $\cO $, black
arrows indicate the action of $\sigma _1$ and dotted arrows indicate the
action of $\sigma _2$, see Figure~\ref{fig:Horbitnotation}.
\begin{figure}[h]
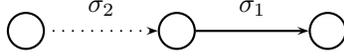

  \begin{center}
  \begin{minipage}{5cm}
  \begin{psmatrix}[mnode=circle,rowsep=.2cm]
  \phantom{x} &\phantom{x} &\phantom{x}
  \end{psmatrix}
  \ncline[linestyle=dotted]{->}{1,1}{1,2}\Aput{$\sigma_2$}
  \ncline{->}{1,2}{1,3}\Aput{$\sigma_1$}
  \end{minipage}
  \caption{The notation for Hurwitz orbits}
  \label{fig:Horbitnotation}
  \end{center}
\end{figure}

For the proof of Proposition~\ref{pro:sizes}
the following theorem going back to Coxeter is useful.

\begin{theorem}
\label{thm:Coxeter} Let $n,p\in \ndN $.
The group $\BG n/(\sigma_{1}^{p})$
is finite if and only if $\frac{1}{n}+\frac{1}{p}>\frac{1}{2}$. In
particular,
\[
\BG 3/\langle\sigma_{1}^{p}\rangle\simeq\begin{cases}
\SG 3 & \text{if }p=2,\\
\mathrm{SL}(2,3) & \text{if }p=3,\\
\mathrm{SL(2,3)\rtimes\mathbb{Z}_{4}} & \text{if }p=4,\\
\mathrm{SL}(2,5)\times\mathbb{Z}_{5} & \text{if }p=5.\end{cases}
\]
\end{theorem}

\begin{proof}
See \cite{Coxeter} for the first claim. For the second claim see
\cite{MR1731872}. The group $\BG n/(\sigma _1)$ can also be identified
with help of \textsf{GAP} \cite{GAP}.
\end{proof}

\begin{prop}
\label{pro:sizes}
Let $d\in \ndN $ and $X$ a braided rack of size $d$. 
Then the possible sizes for a Hurwitz orbit $\cO \subseteq X^3$
are $1$, $3$, $6$, $8$, $9$, $12$, $16$, and $24$.
Two such Hurwitz orbits are isomorphic
if and only if they have the same size.
If $X$ is indecomposable, then
\begin{align*}
\nro{1}&=d,&\nro{3}&=dk_{2},&\nro{6}&=\frac{dt}{6},&\nro{9}&=\frac{d(k_{2}(k_{2}-1)-t)}{3},\\
&&\nro{8}&=\frac{dk_{3}}{2},&\nro{12}&=\frac{dm}{12},&\nro{16}&=\frac{d}{4}(k_{2}k_{3}-k_{2}^{2}+k_{2}+t),
\end{align*}
where
\begin{align}
  \label{eq:mt}
  m=\;&\#\{x\in X\;|\;1\trid x\ne x,\,1\trid^{3}x=x\},\\
  t=\;&\#\{(1,x,y)\;|\;1\trid x=x,\,1\trid y=y,\, x\trid y=y,\, x\ne1,\, y\ne1,\, x\ne y\}
\end{align}
and $1$ is a fixed element of $X$.
\end{prop}

\begin{remark} \label{rem:3divm}
  Let $X$ and $m$ be as in the Proposition. Then $3|m$ since $1\in X$ acts on
  $\{x\in X\,|\,1\trid x\not=x,1\trid ^3x=x\}$ and all orbits of this action
  have size $3$ by assumption.
\end{remark}


\begin{proof}
Let $\cO \subseteq X^3$ be a Hurwitz orbit.
We distinguish two cases and several subcases.

\textit{Case A.}
Assume that $a_1\trid (a_2\trid a_1)=a_2$ for all $(a_1,a_2,a_3)\in \cO $.
Then
$$\sigma _1^3(a_1,a_2,a_3)=(a_1,a_2,a_3) \quad
\text{for all $(a_1,a_2,a_3)\in \cO $.}$$
In particular, $\BG 3/(\sigma _1^3)$ acts on $\cO $ via the Hurwitz action.
The group
$\BG 3/(\sigma _1^3)$ is finite by Theorem~\ref{thm:Coxeter}. Moreover,
the order of $\BG 3/(\sigma _1^3)$ is 24. Thus $\#\cO $ divides 24.
Let $(a,b,c)\in \cO $.
The elements of $\cO $ (counted possibly several times) are
\begin{align*}
A =&\; (a, b, c) &
B =&\; (a\trid b, a\trid c, a) \\
C =&\; (a\trid b, a, c) &
D =&\; ((a\trid b)\trid c, a\trid(b\trid c), a\trid b) \\
E =&\; (b, a\trid b, c) &
F =&\; (a\trid b, c, a\trid c) \\
G =&\; (b, (a\trid b)\trid c, a\trid b) &
H =&\; ((a\trid b)\trid c, a\trid b, a\trid c) \\
I =&\; (a\trid(b\trid c), b, a\trid b) &
J =&\; (b, c, (a\trid b)\trid c) \\
K =&\; (c, (a\trid b)\trid c, a\trid c) &
L =&\; (a\trid(b\trid c), a\trid b, a) \\
M =&\; (c, b\trid c, (a\trid b)\trid c) &
N =&\; (a\trid(b\trid c), a, b) \\
O =&\; (b\trid c, b, (a\trid b)\trid c) &
P =&\; (c, a\trid c, b\trid c) \\
Q =&\; (b\trid c, a\trid(b\trid c), b) &
R =&\; (a, c, b\trid c) \\
S =&\; (a, b\trid c, b) &
T =&\; (b\trid c, (a\trid b)\trid c, a\trid(b\trid c)) \\
U =&\; (a\trid c, a, b\trid c) &
V =&\; (a\trid c, b\trid c, a\trid(b\trid c)) \\
W =&\; ((a\trid b)\trid c, a\trid c, a\trid(b\trid c)) &
X =&\; (a\trid c, a\trid(b\trid c), a),
\end{align*}
see also Figure~\ref{fig:H24} in Appendix B.


\textit{Case A.1.} There exists $(a,b,c)\in \cO $ with $a=b=c$.
Then $\cO =\{(a,a,a)\}$.
There are $\nro{1}\,=\,d$ such orbits.

\textit{Case A.2.} There is $(a,b,c)\in \cO $ with $\#\{a,b,c\}=2$.
By applying $\sigma _1^{-1}$ and/or $\sigma _2^{-1}$ if needed,
we may assume that $a=b$.
In this case, $\cO $ is the Hurwitz orbit of size $8$
depicted in Figure~\ref{fig:H8} in Appendix B, with
\begin{align*}
A =&\; (c,a\trid c, a\trid c) & B =&\; (a, c, a\trid c), & C =&\; (a,a, c) \\
D =&\; (a\trid c,a,a\trid c) & & & E =&\; (a, a\trid c, a) \\
F =&\; (a\trid c, a\trid c, a\trid^2 c) & G =&\; (a\trid c, a\trid ^2 c, a), & H =&\; (a\trid^2 c, a, a)
\end{align*}
Note that $a\trid^2 c$ neither commutes with $a$ nor with $a\trid c$
and it differs from both.
There are $dk_3$ triples $(a_1,a_1,a_3)\in X^3$ with $a_1\trid
a_3\not=a_3$. Since $C$ and $F$ are the only triples in $\cO $ of this type,
we conclude that $\nro{8}\,=\,\frac{1}{2}\,d\,k_3$.

\textit{Case A.3.} Assume that $\#\{a_1,a_2,a_3\}=3$ for all $(a_1,a_2,a_3)\in \cO $.
Then $a\trid b\notin \{a,b,c\}$ and $b\trid c\notin \{a,b,c\}$.
If the triple $A=(a,b,c)$ differs from all other triples in the above list,
then $\#\cO =24$. Otherwise
\begin{align} \label{eq:A=W}
  a=&\;(a\trid b)\trid c, & b=&\;a\trid c, & c=&\;a\trid (b\trid c),
\end{align}
in which case $A=W$ and then
the graph in Figure~\ref{fig:H24} in Appendix B
collapses to the graph in Figure \ref{fig:H12},
corresponding to an orbit of size $12$. The second and third equations in \eqref{eq:A=W}
imply that $b\trid c=a\trid b$, and hence from \eqref{eq:A=W} one obtains that
$c=a\trid (a\trid b)=a\trid ^3 c$. In turn, it follows that \eqref{eq:A=W} is
equivalent to
\begin{align} \label{eq:A=W1}
  b=&\;a\trid c,& c=&\;a\trid ^3 c.
\end{align}
The triples corresponding to the vertices in Figure \ref{fig:H12} are
\begin{align*}
A=&\;(a,a\trid c,c) & B=&\; (a,c,c\trid a) & & \\
& & C=&\; (a\trid c,a,c\trid a)& D =&\; (a\trid c,c\trid a, c)\\
E=&\; (a,c\trid a,a\trid c) & F=&\; (c,a\trid c,c\trid a) & G=&\; (a\trid c,c,a)\\
& & H=&\; (c,a,a\trid c)& I=&\; (c,c\trid a, a)& \\
& & J=&\; (c\trid a, c, a\trid c)& K=&\; (c\trid a, a\trid c, a)& \\
& & L=&\; (c\trid a, a, c).& &
\end{align*}
The number of 12-orbits is just the number of triples $(a_1,a_1\trid a_3,a_3)$
with $a_1\trid a_3\not=a_3$, $a_1\trid ^3a_3=a_3$
(which is $dm$) divided by the number of occurrences of such triples
in the 12-orbit (which is 12), that is, $\nro{12}\,=\,\frac{1}{12}\,d\,m$.

\textit{Case B.} There is $(a,b,c)\in \cO $ such that
two of $a,b,c$ are different but commuting.
We are left with four subcases:
\begin{enumerate}
\item
Two of $a,b,c$ are equal, the third one commutes with both.
\item
$a,b,c$ are pairwise different and commuting.
\item
$a,b,c$ are pairwise different, there are precisely two commuting pairs among $(a,b)$, $(a,c)$, $(b,c)$.
\item
$a,b,c$ are pairwise different, there is precisely one commuting pair.
\end{enumerate}

\textit{Case B.1.} We have an orbit of size $3$, see
Figure~\ref{fig:H3} in Appendix B.
The number of triples of the form $(a_1,a_1,a_3)$
with $a_1\not=a_3$ and $a_1\trid a_3=a_3$ is $\nro{3}\,=\,d\,k_2$.

\textit{Case B.2.} Here $\cO $ is an orbit of size $6$,
see Figure~\ref{fig:H6} in Appendix B.
The braid group acts on the triples in $\cO $ just as the permutation group
$\SG 3$ does.
All $6$ triples of $\cO $ are of this type
and there are $dt$ such triples.
Hence $\nro{6}\,=\,\frac{1}{6}\,d\,t$.

\textit{Case B.3.} By applying $\sigma _1$ and/or $\sigma _2$ if needed,
we may assume that
$a\trid b=b$, $a\trid c=c$. Then $a\trid (b\trid c)=b\trid c$
and $b\trid c\notin \{a,b,c\}$.
Then $\#\cO =9$, see Figure \ref{fig:H9} in Appendix B:
\begin{align*}
A=&\; (b,c,a) & B=&\; (b\trid c,b,a) & C=&\; (c,b\trid c,a) \\
D=&\; (b,a,c) & E=&\; (b\trid c,a,b) & F=&\; (c,a,b\trid c) \\
G=&\; (a,b,c) & H=&\; (a,b\trid c,b) & I=&\; (a,c,b\trid c).
\end{align*}
The total number of triples $(a_1,a_2,a_3)\in X^3$ with
$$
 a_1\trid a_2=a_2,\quad a_1\trid a_3=a_3,\quad
 a_1\not=a_2,\quad a_1\not=a_3,\quad a_2\not=a_3
$$
is $dk_2(k_2\,-\,1)$.
{}From this we subtract the number of triples in which $a_2$
and $a_3$ commute (there are $dt$ such triples)
and divide by the number of occurrences of such triples
in the 9-orbit (which is 3).
Hence $\nro{9}\,=\,\frac{1}{3}\,d\,(k_2(k_2\,-\,1)\,-\,t)$.

\textit{Case B.4.} As argued in Case B.3, we may assume that $a\trid b=b$.
Then
$a\trid (b\trid c)\not=b\trid c$ and
$(a\trid c)\trid (b\trid c) =b\trid c$.
Therefore, the orbit $\cO $
has at most size $16$, with the following triples:
\begin{align*}
A=&\; (a, c, b\trid c) &
B=&\; (a\trid c, a, b\trid c) \\
C=&\; (a\trid c, a\trid(b\trid c), a) &
D=&\; (a, b\trid c, b) \\
E=&\; (c, a\trid c, b\trid c) &
F=&\; (a\trid c, b\trid c, a\trid(b\trid c)) \\
G=&\; (b, a\trid c, a) &
H=&\; (a, b, c) \\
I=&\; (a\trid(b\trid c), b, a) &
J=&\; (b, a, c) \\
K=&\; (c, b\trid c, a\trid c) &
L=&\; (b\trid c, a\trid c, a\trid(b\trid c)) \\
M=&\; (a\trid(b\trid c), a, b) &
N=&\; (b, c, a\trid c) \\
O=&\; (b\trid c, b, a\trid c) &
P=&\; (b\trid c, a\trid(b\trid c), b)
\end{align*}
(see also Figure \ref{fig:H16} in Appendix B).
Further, $\#\{a,b,c,a\trid c,b\trid c\}=5$ and $a\trid (b\trid c)\notin
\{a,b,a\trid c,b\trid c\}$. Looking at the first and last
components of the above triples it follows that $\#\cO=16$.
In particular, $\cO $ did not appear in Cases B1--B3.

The total number of triples $(a_1,a_2,a_3)$ of pairwise different
elements, such that only $a_1$ and $a_2$ commute, can be calculated
as follows: the total number of triples $(a_1,a_2,a_3)$ with pairwise
different $a_1$, $a_2$, $a_3$, such that $a_1$ and $a_2$ commute, but
$a_1$ and $a_3$ do not commute, is $dk_2k_3$.
Among these we have the $d(k_2(k_2\,-\,1)-t)$ triples
with $a_2\trid a_3=a_3$ (see also Case B.3).
With this, the total number of triples, such
that only $a_1$ and $a_2$ commute, is
$$
d\,k_2\,k_3\,-\,d\,(k_2\,(k_2\,-\,1)\,-\,t)
\;\;=\;\; d\big(k_2\,k_3\,-\,k_2^2\,+\,k_2\,+\,t\big).
$$
Finally, there are four triples in $\cO (a,b,c)$
of the form $(a_1,a_2,a_3)$ with
$a_1\trid a_2=a_2$: $F,H,J$ and $L$. Hence
$$
\nro{16} \;\;=\;\;\frac{1}{4}\,d\,\big(k_2\,k_3\,-\,k_2^2\,+\,k_2\,+\,t\big).
$$
This completes the proof of the proposition.
\end{proof}


\subsection{The immunity of a Hurwitz orbit}

Let $X$ be a rack.
In the next section we will need a combinatorial invariant of a Hurwitz orbit
$\cO \subseteq X^3$
which is defined as follows.

\begin{definition}
  Let $\cO \subseteq X^3$ be a Hurwitz orbit. A \textit{quarantine of} $\cO $ is
  a non-empty subset $Q\subseteq \cO $ such that if two of
  $$(x,y,z),\quad (x,y\trid z,y),\quad (x\trid (y\trid z),x,y)$$
  are in $Q$, then the third one is in $Q$.
  Graphically this means the following (see Figures~\ref{fig:Horbitnotation},
  \ref{fig:quarantine}):
  if two vertices along a path consisting of a dotted
  arrow followed by a black arrow are in $Q$, then the third vertex is in $Q$.
  \begin{figure}[h]
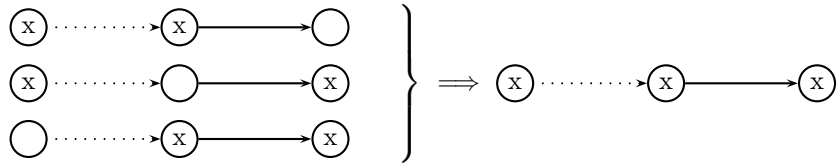

	  \begin{center}
  \begin{minipage}{5cm}
  \begin{psmatrix}[mnode=circle,rowsep=.2cm]
  x&x&\phantom{x} \\
  x&\phantom{x} &x\\
  \phantom{x} &x&x
  \end{psmatrix}
  \ncline[linestyle=dotted]{->}{1,1}{1,2}
  \ncline{->}{1,2}{1,3}
  \ncline[linestyle=dotted]{->}{2,1}{2,2}
  \ncline{->}{2,2}{2,3}
  \ncline[linestyle=dotted]{->}{3,1}{3,2}
  \ncline{->}{3,2}{3,3}
  \end{minipage}
  $\left. \rule[-.8cm]{0pt}{2cm}\right\}$
  $\implies $
  \begin{minipage}{5cm}
  \begin{psmatrix}[mnode=circle]
  x&x&x
  \end{psmatrix}
  \ncline[linestyle=dotted]{->}{1,1}{1,2}
  \ncline{->}{1,2}{1,3}
  \end{minipage}
  \caption{The rule defining a quarantine}
  \label{fig:quarantine}
  \end{center}
  \end{figure}

  A subset $P\subseteq \cO $
  is called a \textit{plague} if the smallest quarantine of $\cO $ containing $P$
  is $\cO $. Let $P$ be a plague of smallest possible size.
  The \textit{immunity of} $\cO $ is the number $\imm{\cO }=\#P/\#\cO \in
  \ndQ \cap (0,1]$.
\end{definition}

\begin{prop} \label{pro:imm}
  Let $X$ be a braided rack and $\cO \subseteq X^3$ a Hurwitz orbit.
  \begin{itemize}
    \item If $\# \cO =1$ then $\imm{\cO}=1$.
    \item If $\# \cO \in \{3,6,9,12\}$ then $\imm{\cO}=\frac{1}{3}$.
    \item If $\# \cO =8$ then $\imm{\cO}=\frac{3}{8}$.
    \item If $\# \cO =16$ then $\imm{\cO}=\frac{5}{16}$.
    \item If $\# \cO =24$ then $\imm{\cO}=\frac{7}{24}$.
  \end{itemize}
\end{prop}

\begin{proof}
  By Proposition~\ref{pro:sizes}, any Hurwitz orbit $\cO \subseteq X^3$ is up to isomorphism
  uniquely determined by its size, which is one of $1$, $3$, $6$, $8$, $9$, $12$, $16$ and $24$.
  The case $\#\cO =1$ is trivial.
  We use a labeling of the triples of the orbit
  as on Figures~\ref{fig:H3}--\ref{fig:H24} in the Appendix.
  If $\#\cO=3$, then $P=\{A\}$ is a plague.
  If $\#\cO=6$, then $P=\{A,B\}$ is a plague and no subset of $\cO $
  of cardinality $1$ is a plague.

  Assume that $\#\cO=8$. The set $\{A,D,H\}$ is a plague of $\cO $. On the
  other hand,
  since $\{A,B,D,E,F,G\}$ and $\{B,C,D,E,G,H\}$ are quarantines,
  for any plague $P$ of $\cO $ we have
  $P\cap \{C,H\}\not=\emptyset $ and $P\cap \{A,F\}\not=\emptyset $.
  Since none of $\{A,C\}$, $\{A,H\}$, $\{C,F\}$, $\{F,H\}$ is a plague,
  we obtain that $\imm{\cO }=3/8$.

  Assume that $\#\cO=9$. The set $\{A,B,C\}$ is a plague of $\cO $. On the
  other hand, $B$ is an element of the quarantines $\{B,C,E,G,H\}$,
  $\{A,B,D,G,I\}$ and $\{B,F\}$ and hence there is no plague $P$ with $B\in
  P$, $\#P=2$. Similarly, $H$ is an element of the quarantines
  $\{B,C,E,G,H\}$, $\{A,C,F,H,I\}$, $\{D,H\}$, and hence there is no
  plague $P$ with $H\in P$, $\#P=2$. Finally, $\{B,C,E,G,H\}$, $\{A,E\}$,
  $\{D,E\}$, $\{E,F\}$, $\{E,I\}$ are quarantines containing $E$, and hence
  there is no plague $P$ with $E\in P$, $\#P=2$. By symmetry, there is no
  plague $P$ of $\cO $ with $\#P=2$. We conclude that $\imm{\cO}=1/3$.

  The proof for the other orbits is similar but more tedious.
  However, the crucial inequality $\imm{\cO }\le \dots $ is easily checked:
  If $\#\cO =12$, then
  $\{A,B,D,E\}$ is a plague.
  If $\#\cO =16$, then $\{A,B,C,E,H\}$ is a plague.
  If $\#\cO =24$, then $\{A,B,C,D,E,K,N\}$ is a plague.
\end{proof}

\section{Nichols algebras over groups}
\label{sec:NA}

For the general theory of Nichols algebras we refer to \cite{MR1913436}.
Details on the relationship between racks and Nichols algebras can be found
in \cite[\S 6]{MR1994219}.

Let $\fie $ be a field.
Yetter-Drinfeld modules over a group $G$ are $\fie G$-modules with a left
coaction $\lcoa :V\to \fie G\otimes V$ satisfying the Yetter-Drinfeld
condition.
Any Yetter-Drinfeld module $V$ over $G$ decomposes as $V=\oplus _{g\in G}V_g$,
where $V_g=\{v\in V\,|\,\lcoa (v)=g\otimes v\}$ for all $g\in G$. The set
\begin{align}\label{eq:supp}
  \supp V=\{g\in G\,|\,V_g\not=0\}
\end{align}
is called the \textit{support of} $V$. By the Yetter-Drinfeld condition,
$\supp V$ is invariant under the adjoint action of $G$.

For any group $G$, any $g\in G$ and any representation $(\rho ,W)$ of
the centralizer $C_G(g)$ of $g$ the $\fie G$-module
\begin{align}
  M(g,\rho )=\fie G \otimes _{\fie C_G(g)}W
\end{align}
is a Yetter-Drinfeld module,
where $W$ is regarded as a $\fie C_G(g)$-module via $\rho \in \End _\fie (W)$ and $\delta
(h\otimes w)=hgh^{-1}\otimes (h\otimes w)$ for all $h\in G$, $w\in W$.
Let $g^G$ be the conjugacy class of $g$ in $G$.
Then $M(g,\rho )=\oplus _{x\in g^G}M(g,\rho )_x$, where
$M(g,\rho )_{hgh^{-1}}=\fie h\otimes W$ for all $h\in G$.

The category $\ydG $ of Yetter-Drinfeld modules over a group $G$ is a braided
monoidal category. Unless otherwise specified,
all tensor products are taken over the fixed field $\fie $.
The braiding is denoted by $c$.
If the braiding appears together with the
tensor product, we also use leg notation: for all $k\in \ndN $,
$i\in \{1,2,\dots ,k-1\}$ and all Yetter-Drinfeld modules $V\sbs1,\dots, V\sbs k$ let
\begin{align*}
 c_{i,i+1}\,&:V\sbs1\otimes \cdots \otimes V\sbs k \to V\sbs 1\otimes \cdots
 \otimes V\sbs{i-1}\otimes
 V\sbs{i+1}\otimes V\sbs i\otimes V\sbs {i+2}\otimes \cdots \otimes V\sbs k,\\
 c_{i,i+1}&=\;\id ^{i-1}\otimes c\otimes \id ^{k-i-1}.
\end{align*}

Nichols algebras are $\ndN _0$-graded braided Hopf algebras.
For any Yetter-Drinfeld module
$V$ over a group $G$ the Nichols algebra of $V$ is denoted by $\toba (V)$.
Then
$$ \toba (V)=\oplus _{n\in \ndN _0}\toba _n(V) $$
is its decomposition into the direct sum of the homogeneous components, where
$\toba _0(V)=\fie $, $\toba _1(V)=V$, and $\toba _n(V)$ is a Yetter-Drinfeld
submodule of $\toba (V)$ for all $n\in \ndN _0$.
The \textit{Hilbert series of} $\toba (V)$ is the formal
power series $\Hilb _{\toba (V)}(t)\in \ndZ [\![t]\!]$ defined by
\begin{align}
  \Hilb _{\toba (V)}(t)=\sum _{i=0}^\infty (\dim \toba _n(V))t^n.
  \label{eq:Hilbert}
\end{align}
We use the notation
\begin{align}
  (n)_{t^r}=& \sum_{i=0}^{n-1}t^{ri},& (\infty )_{t^r}=\sum_{i=0}^\infty t^{ri}
\end{align}
for all $r,n\in \ndN _{\ge 1}$ in connection with the Hilbert series of
Nichols algebras.

\subsection{Nichols algebras with many cubic relations}

The main result of our paper is the following theorem.
In (3) the map $1+c_{12}+c_{12}c_{23}\in \End _\fie (V^{\otimes 3})$
will appear which is defined using leg notation.

\begin{theorem}
  \label{thm:main}
  Let $G$ be a non-abelian group, $g\in G$ and $\rho $ a finite-dimensional
  absolutely irreducible representation of $C_G(g)$. Assume that
  the conjugacy class $X$ of $g$ is a finite braided rack and generates the group $G$.
  Let $V=M(g,\rho )$.
  The following are equivalent.
  \begin{enumerate}
    \item
      The Hilbert series $\Hilb _{\toba (V)}(t)$ of $\toba (V)$
      is a product of factors from
      $$\{(n)_t,(n)_{t^2}\,|\,n\in \ndN _{\ge 2}\cup \{\infty \}\}.$$
    \item $\dim \toba _3(V)\le \dim V\big(\dim \toba _2(V)
      -\frac{1}{3}( (\dim V)^2-1)\big)$.
    \item
      $\dim \ker (1+c_{12}+c_{12}c_{23})\ge
      \frac{1}{3}\dim V( (\dim V)^2-1)$.
    \item The Yetter-Drinfeld module $V$ appears in Tables~\ref{tab:nichols}
      and \ref{tab:candc}.
  \end{enumerate}
\end{theorem}

\begin{remark}
  In the setting of Theorem~\ref{thm:main},
  the rack $X$ is indecomposable
  since $G$ is generated by $X$ and $X$ is a conjugacy class of $G$.
\end{remark}

\begin{definition}
 \label{def:manycubic}
 Let $V$ be a Yetter-Drinfeld module over a group algebra.
 We say that the Nichols algebra $\toba(V)$ \emph{has many cubic relations} 
 if the inequality in Theorem~\ref{thm:main}(3) is satisfied.
\end{definition}

The difficult part of Theorem~\ref{thm:main} is the implication (3)$\Rightarrow $(4).
Its proof will occupy the remaining part of the paper. The other implications are elementary.

\begin{proof}
  (1)$\Rightarrow $(2).
  Consider $\Hilb _{\toba (V)}(t)$ in $\ndZ [\![t]\!]/(t^4)$. Then (1) implies that
  $\Hilb _{\toba (V)}(t)$ is a product of polynomials $1+t$, $1+t+t^2$, $1+t+t^2+t^3$
  and $1+t^2$. By replacing the factors $1+t+t^2$ by $1+t+t^2+t^3$ we may raise
  the coefficient of $t^3$ in $\Hilb _{\toba (V)}(t)$
  without changing the coefficients of $1$, $t$, and $t^2$.
  Now replace the factors $1+t+t^2+t^3$ by $(1+t)(1+t^2)$.
  Thus there exist $n,a,b\in \ndN _0$ such that
  \begin{align}
    \Hilb _{\toba (V)}(t)=(1+t)^a(1+t^2)^b-nt^3\text{ $+$ terms of degree $\ge 4$.}
  \end{align}
  Since $\toba _1(V)=V$, we conclude that $a=\dim V$.
  The coefficient of $t^2$ in $\Hilb _{\toba(V)}(t)$ is $a(a-1)/2+b$ and
  the coefficient of $t^3$ is
  $$\frac{a(a-1)(a-2)}{6}+ab-n=a\left(\frac{a(a-1)}{2}+b\right)-\frac{a(a^2-1)}{3}-n.$$
  This implies the claim.

  (2)$\Rightarrow $(3). Let $S_3=(1+c_{23})(1+c_{12}+c_{12}c_{23})\in \End _\fie (V^{\otimes 3})$
  denote the third quantum symmetrizer. By definition of $\toba _3(V)$ and by (2),
  \begin{align*}
    \dim \ker S_3=&\;(\dim V)^3-\dim \toba _3(V)\\
    \ge &\dim V\left(\dim \ker (1+c)+\frac{1}{3}( (\dim V)^2-1)\right).
  \end{align*}
  On the other hand, by linear algebra we obtain that
\begin{align*}
  & (\dim V)\dim \ker (1+c)+\dim \ker (1+c_{12}+c_{12}c_{23})\\
  & \quad =\dim \ker (1+c_{23}) +\dim \ker (1+c_{12}+c_{12}c_{23})\\
  & \quad \ge \dim \ker (1+c_{23})(1+c_{12}+c_{12}c_{23})\\
  &\quad =\dim \ker S_3.
\end{align*}
  The combination of these two inequalities yields the claim.

  (4)$\Rightarrow $(1) The Hilbert series of $\toba (V)$ can be found in
  Table~\ref{tab:nichols}. For the old examples, $\Hilb _{\toba (V)}(t)$
  was already known.
  For the new examples we calculate $\Hilb _{\toba (V)}(t)$
  in Propositions~\ref{prop:newD3} and \ref{prop:newT}.
\end{proof}

\begin{remark}
  The inequality
  $$\dim \ker S_3 \le (\dim V)\dim \ker (1+c)
  +\dim \ker (1+c_{12}+c_{12}c_{23})$$
  used in the proof of (2)$\Rightarrow $(3) is in fact an
  equality for arbitrary braidings of finite-dimensional vector spaces,
  but we don't need this fact here.
\end{remark}

Let $G$ be a group, $V$ a Yetter-Drinfeld module over $\fie G$
and $X=\supp V$.
For any Hurwitz orbit $\cO \subseteq X^3$ let
$$ V_{\cO }^{\otimes 3}=\oplus _{(x,y,z)\in \cO }V_x\otimes V_y\otimes V_z.$$
Since $V=\oplus _{g\in X}V_g$, we conclude that
$V^{\otimes 3}=\oplus _{\cO }V_{\cO }^{\otimes 3}$,
where $\cO $ is running over all Hurwitz orbits. Further,
each of $V_{\cO}^{\otimes 3}$ is invariant under $1+c_{12}+c_{12}c_{23}$. Thus
\begin{align} \label{eq:dimker}
 \dimker =\sum _{\cO }\rdimker .
\end{align}
The next proposition is one of our main tools to find a good estimate of
the dimension of $\ker (1+c_{12}+c_{12}x_{23})$.

\begin{prop} \label{pro:Oineq}
  Let $G$ be a group, $V$ a non-zero
  finite-dimensional Yetter-Drinfeld module over $\fie G$,
  $X=\supp V$ and $\cO \subseteq X^3$ a Hurwitz orbit.
  Then
  $$\rdimker \le \imm{\cO }\,\dim V_{\cO }^{\otimes 3}.$$
\end{prop}

\begin{proof}
  Let $\tau \in \rker $.
  Then for all $(x,y,z)\in \cO $ there exist uniquely determined
  elements $\tau _{(x,y,z)}\in V_x\otimes V_y\otimes V_z$ such that
  $\tau =\sum _{\bar{x}\in \cO } \tau _{\bar{x}}$.
  Since $\tau \in \ker(1+c_{12}+c_{12}c_{23})$, it follows that
  $$\tau _{(x\trid (y\trid z),x,y)}+c_{12}\tau _{(x,y\trid z,y)}
  +c_{12}c_{23}(\tau _{(x,y,z)})=0$$ for all $(x,y,z)\in \cO $.
  If two summands of such an expression vanish, then so does the third since
  $c_{12}$ and $c_{23}$ are bijective.
  Let now $P\subseteq \cO $ be a plague.
  If $\tau _{(x,y,z)}=0$ for all $(x,y,z)\in P$, then $\tau =0$
  by the choice of $P$.
  Hence the rank of $1+c_{12}+c_{12}c_{23}|_{V_{\cO}^{\otimes 3}}$
  is bounded from below by $\dim V_{\cO }^{\otimes 3}-\#P(\dim V_x)^3$,
  where $x\in X$, that is,
\begin{align*}
\rdimker \le &\#P(\dim V_x)^3
=\frac{\#P}{\#\cO }\dim V_{\cO }^{\otimes 3}=\imm{\cO }\,\dim V_{\cO }^{\otimes 3}.
\end{align*}
This proves the claim.
\end{proof}

%

\begin{definition}
\label{def:optimal}
  Let $G$ be a group, $V$ a non-zero
  finite-dimensional Yetter-Drinfeld module over $\fie G$,
  $X=\supp V$ and $\cO \subseteq X^3$ a Hurwitz orbit.
  The pair $(V,\cO)$ is said to be \emph{optimal with respect to}
  $1+c_{12}+c_{12}c_{23}\in\End _\fie (V^{\otimes3}_\cO)$
  if
  $$\rdimker =\imm{\cO}\,\dim V_\cO^{\otimes3}.$$
\end{definition}

\subsection{Hurwitz orbits with one element}

For the study of Nichols algebras over groups
with many cubic relations,
the Hurwitz orbits of size $1$ and $8$ will play a
distinguished role. We start with a lemma to warm up
and with the analysis of the $1$-orbits.

\begin{lemma}
Let $G$ be a group, $V$ a non-zero Yetter-Drinfeld 
module over $\fie G$, and $X=\supp V$.
Let $q\in \fie \setminus \{0\}$,
$x\in X$, and $\cO=\cO(x,x)\subseteq X^2$. Assume that
$e=\dim V_x<\infty $ and that $xv=qv$ for all $v\in V_x$. Then
$\dim \ker (1+c)$ is the following:
\begin{align*}
\displaystyle{\frac12{e(e+1)}}\quad 	 & \text{ if $q=-1$,}\\
\displaystyle{\frac12{e(e-1)}}\quad 	 & \text{ if $q=1$, $\charf\not=2$,}\\  
0\quad					 & \text{ otherwise.}
\end{align*}
\label{le:warmup}
\end{lemma}

\begin{proof}
Let $v_1,v_2,\dots,v_e$ be a basis of $V_x$. For all $i,j\in\{1,\dots,e\}$ let 
$W_{ij}=\fie(v_i\otimes v_j)$. Decompose $V_x\otimes V_x$ as
\[
V_x\otimes V_x=(\oplus_i W_{ii})\oplus \oplus_{i<j}(W_{ij}\oplus W_{ji}).
\]
Then
\[ (1+c)|_{W_{ii}}=(1+q)\id _{W_{ii}} \]
for all $i\in \{1,\dots,e\}$, and the matrix of $1+c$ with respect to the
basis $(v_i\otimes v_j,v_j\otimes v_i)$ of $W_{ij}\oplus W_{ji}$ for $i\not=j$
is
\[
\begin{pmatrix}
 1 & q \\ q & 1
\end{pmatrix}.
\]
This matrix has rank $1$ if $q^2=1$ and rank $2$ if $q^2\not=1$.
Now the claim of the lemma follows by counting.
\end{proof}

\begin{prop}
\label{pro:orbit1}
Let $G$ be a group, $V$ a non-zero Yetter-Drinfeld 
module over $\fie G$, and $X=\supp V$.
Let $q\in \fie \setminus \{0\}$,
$x\in X$, and $\cO=\cO(x,x,x)\subseteq X^3$. Assume that
$e=\dim V_x<\infty $ and that $xv=qv$ for all $v\in V_x$. Then
$\rdimker $ is the following:
\begin{align*}
\displaystyle{\frac13{e(e^2+2)}}\quad 	 & \text{ if $\charf=3$,
$q=1$,}\\  
\displaystyle{\frac13{e(e^2-1)}}\quad 	 & \text{ if $q=-1$ or
$\charf\not=3$, $q=1$,}\\  
\displaystyle{\frac16{e(e+1)(e+2)}}\quad &
\text{ if $\charf\ne3$, $1+q+q^2=0$,}\\
\displaystyle{\frac16{e(e-1)(e-2)}}\quad &
\text{ if $\charf\ne2,3$, $1-q+q^2=0$,}\\
0\quad											 & \text{ otherwise.}
\end{align*}
In particular, $\rdimker \leq \frac{1}{3}e(e^2+2)$.
\end{prop}

\begin{proof}
Let $v_1,v_2,\dots,v_e$ be a basis of $V_x$. For all $i,j,k\in\{1,\dots,e\}$ let 
$W_{ijk}=\fie(v_i\otimes v_j\otimes v_k)$. Decompose $V_x\otimes V_x\otimes V_x$ as
\[
V_x\otimes V_x\otimes V_x=(\oplus_i W_{iii})\oplus \oplus_{i\ne j}(W_{iij}\oplus
W_{iji}\oplus W_{jii})\oplus(\oplus_{i\ne j\ne k,i\ne k}W_{ijk}).
\]
Then
\begin{align*}
&(1+c_{12}+c_{12}c_{23})(w_1\otimes w_2\otimes w_3)\\
&\qquad
=w_1\otimes w_2\otimes w_3 + q w_2\otimes w_1\otimes w_3+ q^2w_3\otimes w_1\otimes w_2
\end{align*}
for all $w_1\in V_i$, $w_2\in V_j$ and $w_3\in V_k$.
In particular,
if $1+q+q^2=0$, then $\dim \ker (1+c_{12}+c_{12}c_{23})|_{\oplus _iW_{iii}}$ is
$e$, otherwise it is zero.

Assume that $e\geq2$. Let $i,j\in\{1,\dots,e\}$ with $i\ne j$ and let $\lambda_1,\lambda_2,\lambda_3\in\fie$. 
Then
\begin{align*}
(1+c_{12}&+c_{12}c_{23})(\lambda_1 v_i\otimes v_i\otimes v_j+\lambda_2  v_i\otimes v_j\otimes v_i+\lambda_3  v_j\otimes v_i\otimes v_i)\\
&=(\lambda_1+\lambda_1q+\lambda_2q^2)v_i\otimes v_i\otimes v_j\\
&\quad+(\lambda_2+\lambda_3q+\lambda_3q^2)v_i\otimes v_j\otimes v_i+(\lambda_3+\lambda_2q+\lambda_1q^2)v_j\otimes v_i\otimes v_i.
\end{align*}
This expression is zero if and only if
\[
0=(1+q)\lambda_1+q^2\lambda_2
=\lambda_2+(q+q^2)\lambda_3=q^2\lambda _1+q\lambda_2+\lambda_3.
\]
Note that
\[
\det \begin{pmatrix}
1+q & q^2 & 0 \\
0 & 1 & q+q^2\\
q^2 & q & 1
\end{pmatrix}
=(1+q)^2(1-q)^2(1+q+q^2)
\]
and the rank of this matrix is at least 2.
Therefore if $(1+q)(1-q)(1+q+q^2)=0$, then the dimension of
$\ker (1+c_{12}+c_{12}c_{23})$ 
restricted to $\oplus_{i\ne j}(W_{iij}\oplus W_{iji}\oplus W_{jii})$ 
is $e(e-1)$, otherwise it is zero.

Assume that $e\geq3$.
Let $i_1,i_2,i_3\in\{1,\dots,e\}$ be pairwise different elements and 
for all $\sigma\in\SG 3$ let $\lambda_\sigma\in\fie$. 
Similarly to the previous calculation,
\[
\sum_{\sigma\in\SG 3}\lambda_\sigma v_{i_{\sigma(1)}}
\otimes v_{i_{\sigma(2)}}\otimes v_{i_{\sigma(3)}}\in\ker(1+c_{12}+c_{12}c_{23})
\]
if and only if $(\lambda _\sigma )_{\sigma \in \SG 3}\in \ker A$, where
\[ A=
\begin{pmatrix}
1 & 0 & q & q^2 & 0 & 0\\
0 & 1 & 0 & 0 & q & q^2\\
q & q^2 & 1 & 0 & 0 & 0\\
0 & 0 & 0 & 1 & q^2 & q\\
q^2 & q & 0 & 0 & 1 & 0\\
0 & 0 & q^2 & q & 0 & 1
\end{pmatrix}.
\]
We obtain the following facts:
\begin{itemize}
  \item $\det A=(q+1)^4(q-1)^4(q^2+q+1)(q^2-q+1)$.
  \item $\rk A=4$ if and only if $q\in \{-1,1\}$.
  \item $\rk A=5$ if and only if $(q^2+q+1)(q^2-q+1)=0$, $q^2\not=1$.
\end{itemize}
The claim of the proposition follows by summing up
$\dim\ker(1+c_{12}+c_{12}c_{23})$ for different values of $q$.
\end{proof}

\subsection{Hurwitz orbits with eight elements}

The other important Hurwitz orbits for the proof of Theorem~\ref{thm:main}
are the orbits with $8$ elements.

\begin{prop}
\label{pro:orbit8}
Let $G$ be a group, $V$ a non-zero Yetter-Drinfeld 
module over $\fie G$, and $X=\supp V$.
Let $x,y\in X$, $\cO=\cO(x,x,y)\subseteq X^3$,
and $q\in \fie \setminus \{0\}$.
Assume that $x\trid (y\trid x)=y$, $x\not=y$, $e=\dim V_x<\infty $ and $xv=qv$
for all $v\in V_x$. Then $\dim V^{\otimes 3}_{\cO }=8e^3$.
\begin{enumerate}
  \item If $q=-1$ then $\rdimker \le e^2(5e+1)/2$.
  \item If $q\not=-1$ then $\rdimker \le e^2(5e-1)/2$.
\end{enumerate}
\end{prop}

\begin{proof}
Let $z=x\trid y$ and $w=x\trid z$. Then $w\notin\{x,z\}$, $z\trid x=y$,
$w\trid x=z$, $y\trid z=x$, $z\trid w=x$, and 
\[
\cO=\{(x,x,y),(x,z,x),(w,x,x),(z,w,x),(z,z,w),(z,x,z),(y,z,z),(x,y,z)\}.
\]
Since $x\trid (y\trid x)=y$, it follows that $\dim V_x=\dim V_y$ and $\dim
V^{\otimes 3}_{\cO }=8e^3$.
Any element $\tau\in V^{\otimes3}_{\cO}$ has the form
\[
\tau=\tau_{xxy}+\tau_{xzx}+\tau_{wxx}+\tau_{zwx}+\tau_{zzw}+\tau_{zxz}
+\tau_{yzz}+\tau_{xyz},
\]
where $\tau_{ijk}\in V_i\otimes V_j\otimes V_k$ for all $i,j,k\in X$.
Suppose that $\tau \in  \ker(1+c_{12}+c_{12}c_{23})|_{V_{\cO}^{\otimes3}}$.
Applying $1+c_{12}+c_{12}c_{13}$ to $\tau$ and considering summands of different degrees 
we obtain the following equations:
\begin{align*}
&\tau_{xxy}+c_{12}(\tau_{xxy})+c_{12}c_{23}(\tau_{xyz})=0,&\;&\tau_{xzx}+c_{12}(\tau_{zwx})+c_{12}c_{23}(\tau_{zxz})=0,\\
&\tau_{wxx}+c_{12}(\tau_{xzx})+c_{12}c_{23}(\tau_{xxy})=0,&\;&\tau_{zwx}+c_{12}(\tau_{wxx})+c_{12}c_{23}(\tau_{wxx})=0,\\
&\tau_{zzw}+c_{12}(\tau_{zzw})+c_{12}c_{23}(\tau_{zwx})=0,&\;&\tau_{zxz}+c_{12}(\tau_{xyz})+c_{12}c_{23}(\tau_{xzx})=0,\\
&\tau_{yzz}+c_{12}(\tau_{zxz})+c_{12}c_{23}(\tau_{zzw})=0,&\;&\tau_{xyz}+c_{12}(\tau_{yzz})+c_{12}c_{23}(\tau_{yzz})=0.
\end{align*}
This system of equations is equivalent to 
\begin{align}
\label{eq:1}\tau_{zwx}&=-(c_{12}c_{23})^{-1}(1+c_{12})(\tau_{zzw}),\\
\label{eq:2}\tau_{yzz}&=-c_{12}(\tau_{zxz})-c_{12}c_{23}(\tau_{zzw}),\\
\label{eq:4}\tau_{xyz}&=-c_{12}(\tau_{yzz})-c_{12}c_{23}(\tau_{yzz})\\
&=c_{12}(1+c_{23})c_{12}(\tau_{zxz})+c_{12}(1+c_{23})c_{12}c_{23}(\tau_{zzw}),\notag\\
\label{eq:3}\tau_{xzx}&=-(c_{12}c_{23})^{-1}(\tau_{zxz}+c_{12}(\tau_{xyz}))\\
&=-c_{23}^{-1}\left( ( c_{12}^{-1}+c_{12}^2+c_{12}c_{23}c_{12})(\tau_{zxz})+c_{12}(1+c_{23})c_{12}c_{23}(\tau_{zzw}) \right),\notag\\
\label{eq:5}\tau_{wxx}&=-c_{12}(\tau_{xzx})-c_{12}c_{23}(\tau_{xxy}),\\
\label{eq:6}0&=\tau_{xxy}+c_{12}(\tau_{xxy})+c_{12}c_{23}(\tau_{xyz}),\\
\label{eq:7}0&=\tau_{xzx}+c_{12}(\tau_{zwx})+c_{12}c_{23}(\tau_{zxz}),\\
\label{eq:8}0&=\tau_{zwx}+c_{12}(\tau_{wxx})+c_{12}c_{23}(\tau_{wxx}).
\end{align}
Using Equation \eqref{eq:4}, Equation \eqref{eq:6} is
equivalent to 
\begin{gather}
\label{eq:9}
(1+c_{12})(\tau_{xxy})-c_{12}c_{23}c_{12}(1+c_{23})(\tau_{yzz})=0.
\end{gather}
Since $xv=qv$ for all $v\in V_x$, Lemma~\ref{le:warmup} yields that
$\dim \ker(1+c)|_{V_x\otimes V_x}=e(e+1)/2$ if $q=-1$
and $\dim \ker(1+c)|_{V_x\otimes V_x}\le e(e-1)/2$ if $q\not=-1$.
This implies the claim.
\end{proof}

\begin{prop}
\label{pro:orbit8a}
Let $G,V,X,x,y,\cO,q,e$ be as in Proposition~\ref{pro:orbit8}.
Let $v_x\in V_x\setminus\{0\}$, $v_y\in V_y\setminus\{0\}$.
The following are equivalent.
\begin{enumerate}
\item The pair $(V,\cO)$ is optimal with respect to $1+c_{12}+c_{12}c_{23}$.
\item $e=\dim V_x=1$, $q=-1$ and $(1+c^3)(v_x\otimes v_y)=0$. 
\end{enumerate}
\end{prop}

\begin{proof}
We use the same notation as in the proof of Proposition \ref{pro:orbit8}.
Since $\imm{\cO}=3/8$, (1) holds if and only if
Equations \eqref{eq:6}--\eqref{eq:8} 
are satisfied for all tensors $\tau_{xxy}\in V_x\otimes V_x\otimes V_y$,
$\tau_{zxz}\in V_z\otimes V_x\otimes V_z$ and 
$\tau_{zzw}\in V_z\otimes V_z\otimes V_w$, where $\tau _{zwx}$, $\tau _{yzz}$,
$\tau _{xyz}$, $\tau _{xzx}$, $\tau _{wxx}$ are as in
\eqref{eq:1}--\eqref{eq:5}.
By Equation \eqref{eq:2}, Equation~\eqref{eq:9} holds
for all $\tau_{xxy}$, $\tau_{zxz}$ and 
$\tau_{zzw}$ if and only if
\begin{align}
  (1+c)(V_x\otimes V_x)=0,
  \label{eq:1+c}
\end{align}
that is, $\dim V_x=1$ and $q=-1$.

Assume now that Equation~\eqref{eq:1+c} holds. Then $(1+c)(V_u\otimes V_u)=0$
for all $u\in X$. Hence \eqref{eq:1}--\eqref{eq:8} are equivalent to
\begin{align}
  \label{eq:1a}
  \tau _{zwx}=&\;0, \quad \tau _{xyz}=0,\\
  \label{eq:2a}
  \tau _{yzz}=&\;-c_{12}(\tau _{zxz})-c_{12}c_{23}(\tau _{zzw}),\\
  \label{eq:3a}
  \tau _{xzx}=&\;-(c_{12}c_{23})^{-1}(\tau _{zxz}),\\
  \label{eq:4a}
  \tau _{wxx}=&\;c_{12}(c_{12}c_{23})^{-1}(\tau _{zxz})-c_{12}c_{23}(\tau _{xxy}),\\
  \label{eq:5a}
  0=&\;-(c_{12}c_{23})^{-1}(\tau _{zxz})+c_{12}c_{23}(\tau _{zxz}).
\end{align}
Clearly, Equation~\eqref{eq:5a} is equivalent to
\begin{align} \label{eq:6a}
  \tau_{zxz}=(c_{12}c_{23})^2(\tau_{zxz})=c_{12}^2c_{23}c_{12}(\tau_{zxz}).
\end{align}
Since $c_{12}(\tau _{zxz})\in V_y\otimes V_z\otimes V_z$, we conclude that
$c_{23}c_{12}(\tau _{zxz})=-\tau _{zxz}$ and hence Equation~\eqref{eq:6a} is equivalent
to
$$ c_{12}^{-1}(1+c_{12}^3)\tau _{zxz}=0. $$
Since $\dim V_x=1$, this implies the equivalence claimed in the Proposition.
\end{proof}

\begin{prop}
\label{pro:orbit8b}
Let $G$ be a group, $V$ a non-zero Yetter-Drinfeld 
module over $\fie G$, and $X=\supp V$.
Let $x,y\in X$, $\cO=\cO(x,x,y)\subseteq X^3$,
$v_x\in V_x\setminus\{0\}$, $v_y\in V_y\setminus\{0\}$ and $q\in \fie
\setminus \{0,-1\}$.
Assume that $x\trid (y\trid x)=y$, $x\not=y$, $\dim V_x=1$ and $xv=qv$
for all $v\in V_x$.
Then 
$\dim\ker(1+c_{12}+c_{12}c_{23})|_{V_{\cO}^{\otimes3}}\leq 2$
and if equality holds then
$(1+c^3)(v_x\otimes v_y)=0$.
\end{prop}

\begin{proof}
	We use the same notation as in the proof of Proposition \ref{pro:orbit8}. Let $\tau\in\ker(1+c_{12}+c_{12}c_{23})|_{V_{\cO}^{\otimes3}}$
	as in the proof of Proposition \ref{pro:orbit8}. Since $\dim V_x=1$ and $q\ne-1$, we conclude that
	\[
	\tau_{xxy}=c_{12}c_{23}c_{12}(\tau_{yzz})=-c_{12}c_{23}c_{12}^2(\tau_{zxz})-c_{12}c_{23}c_{12}^2c_{23}(\tau_{zzw}),
	\]
	where the first equation follows from \eqref{eq:4} and \eqref{eq:6} and the second from \eqref{eq:2}.
	Hence $\dim\ker(1+c_{12}+c_{12}c_{23})|_{V_{\cO}^{\otimes3}}\leq 2$. 
Equation \eqref{eq:7} implies that 
\begin{equation}
\label{eq:dim2}
0=-c_{23}^{-1}c_{12}^{-1}(1+c_{12}^3)(\tau_{zxz})-c_{23}^{-1}c_{12}^{-1}(1+c_{12}^3)c_{23}(1+c_{12})(\tau_{zzw}).
\end{equation}
Thus, if 
$\dim\ker(1+c_{12}+c_{12}c_{23})|_{V_{\cO}^{\otimes3}}=2$,
then Equation \eqref{eq:dim2} holds for all $\tau_{zxz}$ and $\tau_{zzw}$.
This implies the claim.
\end{proof}

\section{The inequality in the main theorem for braided racks}
\label{sec:ineq}

Let $G$ be a group, $x\in G$, $X$ the conjugacy class of $x$ in $G$,
and let $d\in \ndN $.
Assume that $X$ is a finite indecomposable braided rack of size $d$.
Let $V$ be a finite-dimensional Yetter-Drinfeld module over $G$ with
$\supp V=X$ and let $e=\dim V_x$. Let $q\in \fie \setminus \{0\}$
and assume that $xv=qv$ for all $v\in V_x$.
We collect properties which hold if $\toba (V)$ has many cubic relations.
The number $m$ was defined in Equation~\eqref{eq:mt}.

\begin{prop} \label{prop:verygi}
  Let $d_1,d_8\in \ndN _0$.
  Assume that $\dim \ker (1+c_{12}+c_{12}c_{23})|_{V_{\cO}^{\otimes 3}}\le d_1$
  for all Hurwitz $1$-orbits $\cO \subseteq X^3$ and
  $\dim \ker (1+c_{12}+c_{12}c_{23})|_{V_{\cO}^{\otimes 3}}\le d_8$
  for all Hurwitz $8$-orbits $\cO \subseteq X^3$.
  If $\toba (V)$ has many cubic relations, then
  \begin{align}
    \label{eq:gen_ineq18}
    12k_3d_8+24d_1-k_3^2-30k_3+m-8d^2(e^3-1)+8(e-1) \ge 0.
  \end{align}
\end{prop}

\begin{proof}
  Assume that $\toba (V)$ has many cubic relations.
  Proposition~\ref{pro:Oineq} implies that
  \begin{align*}
    \sum _{\cO\,|\,\#\cO \notin \{1,8\}}\imm{\cO}\dim V_{\cO }^{\otimes 3}
    +\sum _{\cO\,|\,\#\cO \in \{1,8\}}
    \dim \ker (1+c_{12}+c_{12}c_{23})|_{V_{\cO}^{\otimes 3}}\qquad&\\
    \ge \frac{de( (de)^2-1)}{3}.&
  \end{align*}
  Since the only Hurwitz orbits have sizes $1$, $3$, $6$, $8$, $9$, $12$,
  $16$ and $24$, we further obtain that
  \begin{equation}
    \label{eq:Hurwitz_decomposition}
    \nro{1}+3\nro{3}+6\nro{6}+8\nro{8}+9\nro{9}+12\nro{12}+16\nro{16}+24\nro{24}=d^{3}.
  \end{equation}
  Since $X$ is braided, we also know that $k_2=d-k_3-1$.
  Using Proposition~\ref{pro:sizes} and
  the numbers $\imm{\cO}$ from Proposition~\ref{pro:imm}, we conclude that
  the inequality in \eqref{eq:gen_ineq18} holds.
\end{proof}

\begin{lemma} \label{lem:gen_ineq1}
\begin{enumerate}
 \item
 Let $d_1=e(e^2-1)/3$ and $d_8=e^2(5e+1)/2$.
 Then the inequality in \eqref{eq:gen_ineq18} is equivalent to
 $ek_3^2-em-6k_3\le 0$.
 \item
 Let $d_1=\frac{e(e^2+2)}{3}$ and $d_8=\frac{e^2(5e-1)}2$.
 Then the inequality in \eqref{eq:gen_ineq18} is equivalent to
 $e^2 k_3^2-e^2m+6ek_3-24\le 0$.
\end{enumerate}
\end{lemma}

\begin{proof}
  This follows by direct calculation.
\end{proof}

\begin{prop}
\label{pro:gen_ineq}
Assume that $\toba (V)$ has many cubic relations.
Then $k_3\le 6$. Further, if $e\ge 2$ then $k_3\leq 3$.
\end{prop}

\begin{proof}
  Assume first that $q=-1$. Then we can set
  $$d_1=\frac{e(e^2-1)}3,\qquad d_8=\frac{e^2(5e+1)}2$$
  in Proposition~\ref{prop:verygi}
  because of Propositions~\ref{pro:orbit1}, \ref{pro:orbit8}.
  Thus, if $\toba (V)$ has many cubic relations, then
  Proposition~\ref{prop:verygi} implies that
  the inequality in Lemma~\ref{lem:gen_ineq1}(1) holds.
  Hence
  $$(ek_3-6)(k_3-1)+e(k_3-m)\le 6.$$
  Since $e\ge 1$, $m\leq k_3$ and $3|m$ by
  Remark~\ref{rem:3divm}, the latter inequality does not hold for $k_3\ge 7$.
  Similarly, it does not hold if $k_3\ge 4$, $e\ge 2$.

  Assume now that $q\not=-1$. Then, as above, one obtains
  that the inequality in Lemma~\ref{lem:gen_ineq1}(2) holds.
  Since $m\le k_3$, it follows that $e^2k_3(k_3-1)+6ek_3-24\le 0$.
  Since $e\ge 1$, this does not happen for $k_3>3$.
\end{proof}

\section{Braided racks of degree $2$ and $3$-transposition groups}
\label{sec:deg2}

\subsection{$3$-transposition groups}

A set $D$ of involutions in a group $G$ is called
a \emph{set of $3$-transpositions} if $D$ is a union of conjugacy classes of $G$,
$G$ is generated by $D$
and for each $x,y\in D$ the product $xy$ has order $1$, $2$ or
$3$. In this case we say that the pair $(G,D)$ is a \emph{$3$-transposition
group}. For more information related to $3$-transposition groups see \cite{MR1423599}.

\begin{example}
Symmetric groups are $3$-transposition groups, where the $3$-transpo\-sitions
are the transpositions. 
\end{example}

\begin{example}
Let $(G,D)$ be a $3$-transposition group and $\pi :G\to H$ an epimorphism of groups.
Then $(H,\pi (D))$ is a $3$-transposition group.
\end{example}

All $3$-transposition groups generated by at most four elements are
classified in \cite{MR1330798}.
Let $F(k,d)$ be the largest $3$-transposition group $(G,D)$, where
$D$ has size $d$ and
$G$ can be generated by $k$ (and not less that $k$) elements in $D$.

Let $(G,D)$ be a $3$-transposition group and let $Y\subseteq D$ be a subset
generating $D$ as a rack. Let $\mathcal{G}(Y)$ be the graph with vertex set
$Y$ such that $x,y\in Y$ are adjacent in $\mathcal{G}(Y)$ if and only if
$\mathrm{ord}(xy)=3$.

\begin{remark}
  The graph $\mathcal{G}(Y)$ is the complementary graph of the commuting graph
  of $Y$ defined in \cite[Ch.\,2]{MR1423599}.
\end{remark}

One says that two $3$-transposition groups $(G_{1},D_{1})$ and $(G_{2},D_{2})$
\emph{have the same central type} if $G_{1}/Z(G_{1})\simeq G_{2}/Z(G_{2})$
as $3$-transposition groups.

\begin{theorem} \label{thm:3trans}
Let $(G,D)$ be a $3$-transposition group which is generated by a subset
$Y$ of $D$ such that $\#Y\le 3$ and $\mathcal{G}(Y)$ is connected.
Then $G$ has
the same central type as one of the groups $F(1,1)\simeq\mathbb{Z}_{2}$,
$F(2,3)\simeq\SG 3$, $F(3,6)\simeq\SG 4$,
$F(3,9)\simeq\mathrm{SU}(3,2)'$.
\end{theorem}

\begin{proof}
This has been proved independently by several people, see for example
\cite[Theorem 1.1]{MR1330798}.
\end{proof}

\subsection{Graphs and racks of degree two}

\begin{lemma} \label{lem:GYconn}
Let $(G,D)$ be a $3$-transposition group. Assume that $D$ is an indecomposable rack.
Let $Y\subseteq D$ be a minimal subset
generating $D$ as a rack.
Then $\mathcal{G}(Y)$ is connected. 
\end{lemma}

\begin{proof}
Assume that $\mathcal{G}(Y)$ is not connected. Let $Y=Y_1\sqcup Y_2$ be a decomposition
into non-empty disjoint subsets such that
$y_1\trid y_2=y_2$ for all $y_1\in Y_1$, $y_2\in Y_2$. Then
$$D=\langle Y\rangle =\langle Y_1\rangle \cup \langle Y_2\rangle $$
is a decomposition of the rack $D$ into the union of two subracks
and by the minimality of $Y$ we may assume that
$Y_1\cap \langle Y_2\rangle =\emptyset$, $Y_2\cap \langle Y_1\rangle =\emptyset $.
Then $\langle Y_1\rangle \cap \langle Y_2\rangle =\emptyset $, a contradiction to
the indecomposability of $D$ and to Lemma~\ref{lem:rack_indecomposable}.
\end{proof}

\subsection{Examples}

Using the classification of $3$-transposition groups generated by at most
three elements given in Theorem~\ref{thm:3trans}, it is not difficult to produce
examples of braided racks of degree two. 

\begin{example}
\label{example:deg2_size1}
The $3$-transposition group $F(1,1)\simeq\mathbb{Z}_2$
gives the braided rack of one element. 
\end{example}

\begin{example}
\label{example:deg2_size3}
Figure \ref{fig:DD_A2} gives the $3$-transposition group $F(2,3)\simeq\SG 3$. 
The conjugacy class of involutions of $\SG 3$ gives a braided rack isomorphic to
$\mathbb{D}_3$. 
In this case $k_3=2$, see Table \ref{tab:deg2_k3}, and
$\reG X\simeq\SG 3$.
\end{example}

\begin{figure}[h]
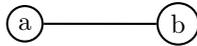

	\begin{center}
\psmatrix[mnode=circle]
a&b
\ncline{1,1}{1,2}
\endpsmatrix
\caption{Diagram of type $(ab)$}
\label{fig:DD_A2}
\end{center}
\end{figure}

\begin{example}
\label{example:deg2_size6}
Figure \ref{fig:DD_A3} gives the $3$-transposition group $F(3,6)\simeq\SG 4$. 
The conjugacy class of transpositions of $\SG 4$ gives a braided 
rack isomorphic to $\mathcal{A}$. In this case $k_3=4$, see Table \ref{tab:deg2_k3}, 
and $\overline{G_X}\simeq\SG 4$.
\end{example}

\begin{figure}[h]
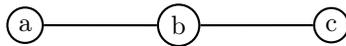

	\begin{center}
\psmatrix[mnode=circle]
a&b&c
\ncline{1,1}{1,2}
\ncline{1,2}{1,3}
\endpsmatrix
\caption{Diagram of type $(abc)$}
\label{fig:DD_A3}
\end{center}
\end{figure}

\begin{example}
  \label{example:deg2_size9}
  Figure \ref{fig:abca} gives the $3$-transposition group $F(3,9)$.
  A presentation for this group is given in \cite{MR1330798}.
  The generators are $a$, $b$ and $c$. The defining relations are
  \begin{gather*}
    a^2=b^2=c^2=(a^bc)^3=1,\\ 
    aba=bab,\;aca=cac,\;bcb=cbc.
  \end{gather*}
  The group $F(3,9)$ has order $54$ and it is isomorphic to $\mathrm{SU}(3,2)'$.
  The elements $a,b,c$ belong to the same conjugacy class $X$.
  The conjugacy class $X$ is a braided rack of $9$ elements. As a rack, $X$ is
  isomorphic to the affine rack $\mathrm{Aff}(\mathbb{F}_9,2)$.
  Further, $k_3=8$.
\end{example}

\begin{figure}[h]
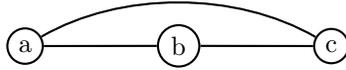

	\begin{center}
  \psmatrix[mnode=circle]
  a&b&c
  \ncline{1,1}{1,2}
  \ncline{1,2}{1,3}
  \ncarc[arcangle=30]{1,1}{1,3}
  \endpsmatrix
  \caption{The diagram $(abca)$}
  \label{fig:abca}
  \end{center}
\end{figure}

\begin{example}
\label{example:deg2_size10}
Figure \ref{fig:DD_A4} gives the $3$-transposition group
$F(4,10)\simeq\SG 5$. 
The conjugacy class of transpositions of $\SG 5$ gives a braided 
rack isomorphic to $\mathcal{C}$. In this case $k_3=6$, see Table \ref{tab:deg2_k3}, and
$\overline{G_X}\simeq\SG 5$.
\end{example}

\begin{figure}[h]
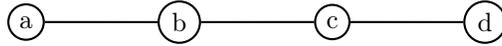

	\begin{center}
\psmatrix[mnode=circle]
a&b&c&d
\ncline{1,1}{1,2}
\ncline{1,2}{1,3}
\ncline{1,3}{1,4}
\endpsmatrix
\caption{Diagram of type $(abcd)$}
\label{fig:DD_A4}
\end{center}
\end{figure}

\begin{example}
  \label{example:deg2_size12}
  Figure \ref{fig:DD_D4} gives the $3$-transposition group $F(4,12)$.
  Following \cite{MR1330798},
  the group $F(4,12)$ is defined by generators $a$, $b$, $c$, $d$ and
  relations
  \begin{gather*} 
    a^2=b^2=c^2=d^2,\\
    aba=bab,\;ada=dad,\;aca=cac,\\
    cb=bc,\;cd=dc,\;bd=db.
  \end{gather*}
  The group $F(4,12)$ has order $192$.
  The elements $a,b,c,d$ belong to the
  same conjugacy class $X$. The conjugacy class $X$ is a braided rack of size
  $12$ and $k_3=8$.
\end{example}

\begin{figure}[h]
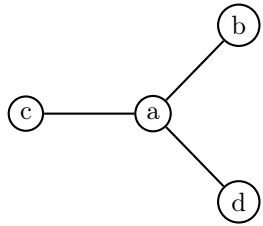

	\begin{center}
  \psmatrix[colsep=0.6,rowsep=0.6,mnode=circle]
  &&&b\\
  c&&a\\
  &&&d
  \ncline{2,1}{2,3}
  \ncline{2,3}{1,4}
  \ncline{2,3}{3,4}
  \endpsmatrix
  \caption{Diagram of type $(cab,ad)$.}
  \label{fig:DD_D4}
  \end{center}
\end{figure}


Let $\mathcal{GC}$ denote the category of pairs $(G,D)$, where $G$ is a group with trivial center,
$D$ is a conjugacy class of $G$ generating $G$, and a morphism between pairs $(G,D)$ and $(H,E)$
is a group homomorphism $f:G\to H$ such that $f(D)=E$.

\begin{prop} \cite[Prop.\,3.2]{MR1994219}
\label{prop:ecAG}
There is an equivalence of categories between
the category of faithful indecomposable racks with surjective morphisms and
the category $\mathcal{GC}$.
\end{prop}

\begin{corollary}
\label{cor:deg2_3trans}
There is an equivalence of categories between
the category of braided indecomposable racks of degree two
with surjective morphisms and
the category of $3$-transposition groups in $\mathcal{GC}$.
\end{corollary}

\begin{proof}
 Let $\Gamma $ denote the equivalence in Proposition~\ref{prop:ecAG}.
 Then a rack $X$ has degree two if and only if $D$ consists of involutions,
 where $\Gamma (X)=(G,D)$. Further, $X$ is braided if and only if
 $\mathrm{ord}(xy)\in \{1,2,3\}$ for all $x,y\in D$.
\end{proof}

\begin{lemma} \label{lem:k3}
  Let $X,X'$ be finite indecomposable braided racks such that $X\subsetneq X'$.
  Then $k_3(X)<k_3(X')$.
\end{lemma}

\begin{proof}
  Since $X'$ is indecomposable, there exist $x\in X$, $y\in X'\setminus X$
  such that $x\trid y\not=y$. Then
  $k_3(X')=\#\{z\in X'\,|\,x\trid z\not=z\}
   >\#\{z\in X\,|\,x\trid z\not=z\}=k_3(X)$.
\end{proof}

\begin{prop} \label{prop:deg2}
Let $X$ be a finite braided indecomposable rack of degree two with $k_3\le 6$.
Then $X$ is isomorphic to
one of the racks $\mathbb{D}_3$, $\mathcal{A}$ and $\mathcal{C}$.
\end{prop}

\begin{proof}
First assume that the rack $X$ is generated by at most three elements.
By Theorem~\ref{thm:3trans} and Corollary~\ref{cor:deg2_3trans} we only
have to check Examples~\ref{example:deg2_size3},
\ref{example:deg2_size6}, and \ref{example:deg2_size9}. In this case
$X\simeq \mathbb{D}_3$ or $X\simeq\mathcal{A}$.
Assume now that $X$ is generated by a subset $Y\subseteq X$
with $\#Y=4$. By Lemma~\ref{lem:GYconn}, the graph $\mathcal{G}(Y)$ is connected.
If $\mathcal{G}(Y)$ contains a triangle, then $k_3(X)>8$ by
Lemma~\ref{lem:k3} and Example~\ref{example:deg2_size9}.
If $\mathcal{G}(Y)$ is as in Example~\ref{example:deg2_size12} then
$k_3(X)>6$. Hence $X\simeq \mathcal{C}$ by Example~\ref{example:deg2_size10}.
Finally, if $X$ is generated by more than four elements, then $k_3(X)>6$ by
Lemma~\ref{lem:k3}.
\end{proof}

\begin{table*}
\begin{center}
\begin{tabular}{|c|c|c|c|c|}
\hline 
Rack 	& Diagram 		& Size & $k_3$ & Reference \\\hline
$\mathbb{D}_3$	& $(ab)$ 	& $3$  & $2$	 & Example \ref{example:deg2_size3}\\\hline
$\mathcal{A}$	& $(abc)$	& $6$	 & $4$& Example \ref{example:deg2_size6}\\\hline
$\mathrm{Aff}(9,2)$ & $(abca)$ & $9$  & $8$& Example \ref{example:deg2_size9}\\\hline
$\mathcal{C}$	& $(abcd)$	& $10$ & $6$& Example \ref{example:deg2_size10}\\\hline
\end{tabular}
\caption[aboveskip]{Some braided racks of degree two}
\label{tab:deg2_k3}
\end{center}
\end{table*}

\section{Braided racks of degree four}
\label{sec:deg4}

\begin{prop} \label{prop:deg4}
Let $X$ be a finite braided indecomposable rack of degree $4$ such that $k_3\le 6$.
Then $X$ is isomorphic to $\mathcal{B}$.
\end{prop}

\begin{proof}
Let $1,2,\dots ,\#X$ denote the elements of $X$.
Since $k_{3}$
is the number of moved points of the permutation $\varphi_{1}$, the
type of $\varphi_{1}$ is 
$(2,4)$ or $(4)$.

\textit{Type $(2,4)$.}
Without loss of generality we may assume that 
\[
\varphi_{1}=(2\,3)(4\,5\,6\,7).
\]
Lemma \ref{lem:braided_cycle} implies that $\varphi_{2}=(1\,3)\pi_{2}$, where $\pi_{2}$
is a $4$-cycle that commutes with $(1\,3)$. Similarly, $\varphi_{3}=(1\,2)\pi_{3}$,
where $\pi_{3}$ is a $4$-cycle that commutes with $(1\,2)$. We prove that $2\trid 4\notin \{4,5,6,7,8\}$
which is a contradiction.

Assume that $2\trid4=4$.
Then $1\trid(2\trid5)=\varphi_{2}\varphi_{1}(2\trid4)=\varphi_{2}\varphi_{1}(4)=2\trid5$.
Let $8=2\trid5$ be this new element that commutes with
$1$. Then $8=\varphi_{1}^{2}(2\trid5)=2\trid7$,
which is a contradiction. 

Assume that $2\trid4=5$. Then $4\trid 5=1$ and $4\trid 5=2$ by Lemma~\ref{lem:braided_cycle}
which is a contradiction.

Assume that $2\trid4=6$.
Then $2\trid6=\varphi_{1}^{2}(2\trid4)=\varphi_{1}^{2}(6)=4$,
which contradicts the type of $\varphi_{2}$. 

Assume that $2\trid4=7$. Then $2\trid 6=\varphi _1^2(2\trid 4)=\varphi_1^2(7)=5$.
we obtain that $\varphi_{2}=(1\,3)(4\,7\,6\,5)$ and $\varphi_{3}=(1\,2)(5\,4\,7\,6)$.
Then $2\trid 7=6$ implies that $6\trid2=7$
and $3\trid 7=6$ implies that $6\trid3=7$,
a contradiction.

Assume that $2\trid4=8$.
Then $8=\varphi_{1}^{2}(8)=\varphi_{1}^{2}(2\trid4)=2\trid6$,
which is a contradiction.

\textit{Type $(4)$.}
Without loss of generality we may assume that 
\[
\varphi_{1}=(2\,3\,4\,5).
\]
Then $1\trid 5=2$, $5\trid 2=1$, and hence $5$ and $2$ do not commute.
Let $x=2\trid5$. Then Lemma \ref{lem:braided_cycle} implies that 
$\varphi_{2}=(3\,1\,5\, x)$, $\varphi_{3}=(4\,1\,2\,\varphi_{1}(x))$,
$\varphi_{4}=(5\,1\,3\,\varphi_{1}^{2}(x))$ and $\varphi_{5}=(2\,1\,4\,\varphi_{1}^{3}(x))$.

Assume that $2\trid5=4$. Then $3\trid 2=\varphi_{1}(2\trid 5)=\varphi _1(4)=5$
and hence $2\trid 5=3$, a contradiction.
Therefore $2\trid5=6$ and hence $\varphi_{2}=(3\,1\,5\,6)$, $\varphi_{3}=(4\,1\,2\,6)$,
$\varphi_{4}=(5\,1\,3\,6)$, $\varphi_{5}=(2\,1\,4\,6)$ and $\varphi_{6}=(2\,5\,4\,3)$.
Therefore $X\simeq\mathcal{B}$, the rack associated to the conjugacy
class of $4$-cycles in $\SG 4$.
\end{proof}

\section{Braided racks of degree three or six}
\label{sec:deg36}

\begin{prop} \label{prop:deg3}
Let $X$ be a finite braided indecomposable rack of degree $3$ such that $k_3\le 6$.
Then $X$ is isomorphic to the rack $\mathcal{T}$. 
\end{prop}

\begin{proof}
Let $1,2,\dots ,\#X$ denote the elements of $X$.
Since $k_{3}$ is the number of moved points of the permutation $\varphi_{1}$, the
type of $\varphi_{1}$ is 
$(3)$ or $(3,3)$.

\textit{Type $(3)$.}
Without loss of generality we may assume that 
\[
\varphi_{1}=(2\,3\,4).
\]
Lemma \ref{lem:braided_cycle} implies that 
$\varphi_2=(3\,1\,4)$, $\varphi_3=(4\,1\,2)$ and $\varphi_4=(1\,3\,2)$.
Then $X\simeq\mathcal{T}$.

\textit{Type $(3,3)$.}
Without loss of generality we may assume that
\[
\varphi_{1}=(2\,3\,4)(5\,6\,7).
\]
Lemma \ref{lem:braided_cycle} implies that
$\varphi_2$ contains the $3$-cycle $(3\,1\,4)$, $\varphi_5$ contains the
$3$-cycle $(6\,1\,7)$ and $\varphi_7$ contains the $3$-cycle $(1\,6\,5)$.

If $\varphi _2$ contains the $2$-cycle $(5\,6\,7)$ or $(5\,7\,6)$ then
$2\trid5\in\{6,7\}$. However, $2\trid5=6$ and Lemma \ref{lem:braided_cycle}
imply that $1=5\trid6=2$, a contradiction. Similarly, $2\trid5=7$ and Lemma
\ref{lem:braided_cycle} imply that $6=5\trid7=2$, a contradiction. 

Without loss of generality we may assume that $2\trid5=5$. Apply the 
permutation $\varphi_2\varphi_1$ to
$2\trid5=5$ to obtain $1\trid(2\trid6)=2\trid6$. We may assume that $8=2\trid6$
and that $2\trid 8\in \{7,9\}$.

Assume that $2\trid8=9$.
Applying $\varphi_2\varphi_1$
we obtain that $1\trid(2\trid8)=2\trid9$, that is, $9=2\trid 9$.
This is a contradiction to $2\trid 8=9$. 

We have proved that $2\trid 8=7$ and hence $\varphi_2=(3\,1\,4)(6\,8\,7)$.
Since $2\trid7=6$, Lemma \ref{lem:braided_cycle} implies that 
$5=7\trid6=2$, a contradiction.
\end{proof}

\begin{prop} \label{prop:deg6}
Let $X$ be a finite braided indecomposable rack of degree $6$ such that $k_3\le 6$.
Then $X$ is isomorphic to one of the racks $\mathrm{Aff}(7,3)$, $\mathrm{Aff}(7,5)$. 
\end{prop}

\begin{proof}
Let $1,2,\dots ,\#X$ denote the elements of $X$.
Since $k_{3}$ is the number of moved points of the permutation $\varphi_{1}$, the
type of $\varphi_{1}$ is 
$(2,3)$ or $(6)$.

\textit{Type $(2,3)$.}
\label{type23}
Without loss of generality we may assume that 
\[
\varphi_{1}=(2\,3)(4\,5\,6).
\]
Lemma \ref{lem:braided_cycle} implies that
$\varphi_2$ contains the transposition $(1\,3)$ and 
$\varphi_4$ contains the $3$-cycle $(1\,6\,5)$.

First we show that $2\trid4=4$.
Indeed, the possible values for $2\trid4$ are $4$, $5$, $6$ and $7$.
The case $2\trid4=7$ is excluded by the
formula $\varphi_1^2(2\trid4)=2\trid6$. The case $2\trid4=5$ contradicts 
Lemma \ref{lem:braided_cycle} since $1\trid4=5$. If $2\trid4=6$, 
then Lemma \ref{lem:braided_cycle}
implies that $2=4\trid6=5$, a contradiction.

Since $2\trid4=4$, we obtain that
$2\trid5=\varphi_1^4(2\trid4)=\varphi_1^4(4)=5$ and
$2\trid6=\varphi_1^2(2\trid4)=\varphi_1^2(4)=6$.
Since the permutation 
$\varphi_2$ is of type $(2,3)$, we may assume that $\varphi _2=(1\,3)(7\,8\,9)$. 
Then 
$8=1\trid8=\varphi_2\varphi_1(2\trid7)=\varphi_2\varphi_1(8)=2\trid8$, which is a contradiction.

\textit{Type $(6)$.}
Without loss of generality we may assume that 
\[
\varphi_{1}=(2\,3\,4\,5\,6\,7).
\]
Lemma \ref{lem:braided_cycle} implies that 
$\varphi_2=(3\,1\,7\,\cdots )$, $\varphi_3=(4\,1\,2\,\cdots)$,
$\varphi_4=(5\,1\,3\,\cdots)$, $\varphi_5=(6\,1\,4\,\cdots)$, 
$\varphi_6=(7\,1\,5\,\cdots)$ and 
$\varphi_7=(2\,1\,6\,\cdots)$. Since
\[
7 = 2\trid1 = 2\trid(3\trid4) = (2\trid3)\trid(2\trid4) = 1\trid(2\trid4),
\]
it follows that $2\trid4 = 6$. Moreover, $2\trid5\ne5$. Indeed, otherwise
\[
7 = 2\trid1 = 2\trid(4\trid5) = (2\trid4)\trid(2\trid5) = 6\trid5\ne7,
\]
a contradiction. Therefore $\varphi_2\in\{(3\,1\,7\,4\,6\,5), (3\,1\,7\,5\,4\,6)\}$. 
By conjugation with $\varphi_1$ one obtains all permutations $\varphi_i$ 
with $i\in\{3, 4, 5, 6, 7\}$. Since $X$ is indecomposable, we conclude that $\#X=7$.

Assume that $\varphi_2 = (3\,1\,7\,4\,6\,5)$.
Then $\varphi_3 = (4\,1\, 2\, 5\, 7\, 6)$, $\varphi_4=(5\, 1\, 3\, 6\, 2\, 7)$,
$\varphi_5=(6\, 1\, 4\, 7\, 3\, 2)$, $\varphi_6 = (7\, 1\, 5\, 2\, 4\, 3)$ 
and $\varphi_7 = (2\, 1\, 6\, 3\, 5\, 4)$. This rack is isomorphic to the
affine rack $\mathrm{Aff}(7, 5)$.
On the other hand, if $\varphi_2 = (3\, 1\, 7\, 5\, 4\, 6)$,
then $\varphi_3 = (4\, 1\, 2\, 6\, 5\, 7)$, $\varphi_4 = (5\, 1\, 3\, 7\, 6\, 2)$,
$\varphi_5 = (6\, 1\, 4\, 2\, 7\, 3)$, $\varphi_6 = (7\, 1\, 5\, 3\, 2\, 4)$ and 
$\varphi_7 = (2\, 1\, 6\, 4\, 3\, 5)$. This rack is isomorphic to the
affine rack $\mathrm{Aff}(7, 3)$.
\end{proof}

\section{The proof of Theorem~\ref{thm:main}}
\label{sec:proof}

In this section we prove Theorem~\ref{thm:main}(3)$\Rightarrow $(4). If $\#X=1$,
then $G$ is cyclic.
Hence $\#X>1$. Since $X$ is indecomposable, Proposition~\ref{pro:degree} implies that
the degree of $X$ is $2$, $3$, $4$, or $6$. Further, $k_3\le 6$ by Proposition~\ref{pro:gen_ineq}
and $k_3\le 3$ if the degree of $\rho $ is at least $2$.
By Propositions~\ref{prop:deg2}, \ref{prop:deg4}, \ref{prop:deg3} and \ref{prop:deg6}
we only have to take care about the racks $X=\mathbb{D}_3$, $\mathcal{T}$, $\mathcal{A}$, $\mathcal{B}$,
$\mathcal{C}$, $\mathrm{Aff}(7,3)$ and $\mathrm{Aff}(7,5)$. Each of these racks is considered in a separate subsection.
Since $G$ is generated by $X$, there is an epimorphism $G_X\to G$. Thus we may assume that $G=G_X$.
The elements of $X$ and their image in $G_X$ will be denoted by $1,2,\ldots ,\#X$ and $x_1,x_2,\ldots ,x_{\#X}$,
respectively. Since any braided rack is faithful, the elements $x_1,\dots x_{\#X}$ are pairwise distinct.

During the proof some known and some new finite-dimensional Nichols algebras will appear.
The Hilbert series of these algebras are collected in Table~\ref{tab:nichols}. The formulas
for the known examples are taken from \cite[Table\,1]{Marburg}.


\subsection{The rack $\mathbb{D}_3$}
\label{sec:D3}
Let $X=\{1,2,3\}=\mathbb{D}_3$. The size of $X$ is $d=3$.
The rack structure of $X$ is uniquely determined by $\varphi _1=(2\,3)$.

\begin{lemma} \cite[Lemma 5.2]{Marburg}
\label{lem:b3_centralizer}
The centralizer of $x_1$ in $G_X$ is the cyclic group generated by $x_1$.
\end{lemma}

\begin{prop}
\label{prop:b3}
Let $\rho$ be an absolutely irreducible representation of $C_{G_X}(x_1)$ and let
$V=M(x_1,\rho)$. Then $\toba(V)$ has many cubic relations if and only if
$\rho(x_1)=-1$ or $\charf=2$, $\rho(x_1)^2+\rho(x_1)+1=0$.
\end{prop}

\begin{remark}
  The Nichols algebra $\toba (V)$ in the case $\rho (x_1)=-1$ appeared first in
  \cite{MR1800714}. Some data about $\toba (V)$ can be found in Table~\ref{tab:nichols}.
  The Nichols algebra $\toba (V)$ in the case $\charf =2$,
  $\rho(x_1)^2+\rho(x_1)+1=0$ is an unpublished example found by H.-J.~Schneider
  and the first author. More details can be found in Proposition~\ref{prop:newD3}.
\end{remark}

\begin{proof}
Assume first that $\rho(x_1)=-1$ or $\charf=2$,
$\rho(x_1)^2+\rho(x_1)+1=0$. Then $\Hilb _{\toba (V)}(t)$ is a product of polynomials
$(n)_t$ and $(n)_{t^2}$ for some $n\in \ndN $, see Table~\ref{tab:nichols}.
We conclude that $\toba(V)$ has many cubic relations
by Theorem \ref{thm:main}(4)$\Rightarrow$(3).

Assume that $\toba(V)$ has many cubic relations and $\rho(x_1)\ne-1$. 
By Lemma~\ref{lem:b3_centralizer}, the group $C_{G_X}(x_1)$ is abelian. 
Hence the degree of $\rho$ is $e=1$. Further, Proposition~\ref{pro:orbit8} implies that
$\rdimker\leq2$ for all orbits $\cO$ of size $8$, since $\rho(x_1)\ne-1$. 
If $\rdimker\leq1$ for all orbits $\cO$ of size $8$, then 
Proposition~\ref{prop:verygi} yields a contradiction since $d=3$, $m=0$, $k_3=2$ and $d_1\le 1$.
Since the three Hurwitz orbits of size $8$ are conjugate, 
we conclude that $\rdimker=2$ for all orbits $\cO$ of size $8$. 
Proposition \ref{pro:orbit8b} implies that $(1+c_{12}^3)(v\otimes x_3v)=0$ for all $v\in V_{x_1}$.
Then
\begin{align*}
0=(1+c_{12}^3)(v\otimes x_3v)&=v\otimes x_3v+x_2x_1x_3v\otimes x_3v\\
&=(v+x_2^2x_1v)\otimes x_3v=2v\otimes x_3v
\end{align*}
since $x_1^2=x_2^2$. Therefore $\charf=2$. If $\rho(x_1)^2+\rho(x_1)+1\ne0$, then 
Proposition~\ref{pro:orbit1} gives that
$\rdimker=0$ for all orbits $\cO$ of size $1$.
Then Proposition~\ref{prop:verygi} yields a contradiction.
\end{proof}

Now we discuss one of the Nichols algebras mentioned above.
Assume that $\charf = 2$ and that $\fie $ contains an
element $q\in\fie $ with $q^2+q+1=0$. Recall that
$X=\mathbb{D}_3$. Let $\rho$ be the absolutely irreducible representation of $C_{G_X}(x_1)$ with
$\rho (x_1)=q$.
Let $V=M(x_1,\rho)$, $a\in V_{x_1}\setminus \{0\}$, $b=q ^{-1}x_3 a$ and $c=q^{-1}x_1 b$.
The action of $G_X$ on $V$ is then determined by Table~\ref{tab:GXactionD}.

\begin{table}
\begin{equation*}
 \begin{array}{c|ccc}
   & a & b & c\\ \hline
   x_1 & q a & q c & q b\\
   x_2 & q c & q b & q a\\
   x_3 & q b & q a & q c
 \end{array}
\end{equation*}
\caption{The action of $G_X$ on $V$, where $X=\mathbb{D}_3$.}
\label{tab:GXactionD}
\end{table}

\begin{prop} \label{prop:newD3}
The Nichols algebra $\toba (V)$
can be presented by generators $a,b,c$ with defining relations
\begin{gather}
\label{eq:D3_rel1}  ab+q^2bc+q ca=0,\\
\label{eq:D3_rel2}  ac+q^2cb+q ba=0,\\
\label{eq:D3_rel3}  a^3=b^3=c^3=0,\\
\label{eq:D3_rel4}  (a^2b^2)^3+b(a^2b^2)^2a^2b+b^2(a^2b^2)^2a^2+ab^2(a^2b^2)^2a=0.
\end{gather}
The Hilbert series of $\toba (V)$ is
\begin{align*}
\Hilb _{\toba (V)}(t)=&(3)_t(4)_t(6)_t(6)_{t^2}.
\end{align*}
The dimension of $\toba (V)$ is $432$. The top degree of $\toba (V)$ is $20$.
An integral of $\toba (V)$ is given by 
\[
a^2ba^2b(a^2b^2)^3c^2.
\]
\end{prop}

\begin{proof}
The relations in
\eqref{eq:D3_rel1}--\eqref{eq:D3_rel4}
``generate'' a Hopf ideal of the tensor algebra $T(V)$.
Using the theory of Gr\"obner bases \cite{GBNP, GAP}, it can be seen that
the quotient algebra has the stated dimensions in each degree. Using \cite[Theorem 6.4 part (2)]{MR1994219},
it is sufficient to see that 
$a^2ba^2b(a^2b^2)^3c^2$
does not vanish in $\toba (V)$
in order to prove the claim. Direct calculation gives that
\[
\partial_b\partial_b\partial_a\partial_a\partial_c\partial_c
\partial_a\partial_a\partial_c\partial_c\partial_b\partial_b
\partial_c\partial_b\partial_c\partial_b\partial_c\partial_b
\partial_c\partial_c
\]
applied to $a^2ba^2b(a^2b^2)^3c^2$ gives a non-zero number.
This completes the proof. 
\end{proof}

\subsection{The rack $\mathcal{T}$}
\label{sec:T}

Let $X=\{1,2,3,4\}=\mathcal{T}$ and $d=4$.
Using that $X$ is braided, the rack structure of $X$ is uniquely
determined by $\varphi _1=(2\,\,3\,\,4)$. Note that $x_1\trid x_2=x_3$
in contrast to the convention $x_2\trid x_1=x_3$ in \cite[\S 5.2]{Marburg}.
Hence our group $G_X$ is the opposite of the group $G_X$ in \cite[\S 5.2]{Marburg}.

\begin{lemma}
\label{lem:b4_centralizer}
\cite[Lemma 5.5]{Marburg}.
The centralizer of $x_1$ in $G_X$ is abelian and is generated by $x_1$
and $x_2x_4$. Further, the relation $(x_2x_4)^2=x_1^4$ holds in $G_X$.
\end{lemma}

\begin{lemma}
\label{lem:Torbit8}
Let $x,y,x',y'\in X$ with $x\ne y$ and $x'\ne y'$. Then
$\cO(x,x,y)$ and $\cO(x',x',y')$ are conjugate.
\end{lemma}

\begin{proof}
 By applying $\varphi_1$ we conclude that $\cO(1,1,2)$ is conjugate to
 $\cO(1,1,z)$ for all $z\in X\setminus \{1\}$.
 Since $X$ is indecomposable, the claim follows.
\end{proof}

\begin{prop}
\label{prop:b4}
Let $\rho$ be an absolutely irreducible representation of $C_{G_X}(x_1)$ and let
$V=M(x_1,\rho)$. Then $\toba(V)$ has many cubic relations if and only if
\begin{enumerate}
\item $\rho(x_1)=-1$ and $\rho(x_2x_4)=1$, or
\item $\rho(x_1)^2+\rho(x_1)+1=0$ and $\rho(x_2x_4)=-\rho(x_1)^{-1}$.
\end{enumerate}
\end{prop}

\begin{remark}
  The Nichols algebra $\toba (V)$ with $\rho $ as in (1) appeared first in
  \cite[Thm.\,6.15]{MR1994219}. For arbitrary fields the example was discussed in
  \cite[Prop.\,5.6]{Marburg}. Recall that $\toba (V)$ depends essentially on $\charf $.

  The Nichols algebra $\toba (V)$ with $\rho $ as in (2) is new. It will be discussed
  in Proposition~\ref{prop:newT}.
\end{remark}

\begin{proof}
Assume first that (1) or (2) hold. Then $\Hilb _{\toba (V)}(t)$ is a product of polynomials
$(n)_t$ and $(n)_{t^2}$ for some $n\in \ndN $, see Table~\ref{tab:nichols}.
We conclude that $\toba(V)$ has many cubic relations
by Theorem \ref{thm:main}(4)$\Rightarrow$(3).

Assume that $\toba(V)$ has many cubic relations.
By Lemma~\ref{lem:b4_centralizer}, the group $C_{G_X}(x_1)$ is abelian. 
Hence the degree of $\rho$ is $e=1$. Since $m=k_3=3$, Proposition~\ref{prop:verygi}
implies that $36d_8+24d_1\ge 96$, where $d_1\in \{0,1\}$ and $d_8\in \{0,1,2,3\}$
by Proposition~\ref{pro:orbit8}. Hence $d_8=3$, $d_1\in \{0,1\}$ or $d_8=2$, $d_1=1$.

Assume first that $\rho (x_1)=-1$. Then we can choose $d_1=0$ by Proposition~\ref{pro:orbit1}.
Since then $d_8=3$, there is at least one $8$-orbit with immunity $3/8$.
By Lemma~\ref{lem:Torbit8}, all Hurwitz orbits of size $8$ are conjugate.
Hence for each $8$-orbit the pair $(V,\cO )$ is optimal
with respect to $1+c_{12}+c_{12}c_{23}$. Thus Proposition~\ref{pro:orbit8a}
implies that
\begin{equation}
\begin{aligned} \label{eq:444}
0=(1+c_{12}^3)(v\otimes x_3v)&=v\otimes x_3v+x_2x_1x_3v\otimes x_3v\\
&=(v+x_2x_4x_1v)\otimes x_3v=(1+\rho (x_2x_4)\rho (x_1))v\otimes x_3v
\end{aligned}
\end{equation}
for all $v\in V_{x_1}$. Since $\rho (x_1)=-1$,
it follows that $\rho (x_2x_4)=1$, that is, (1) holds.

Assume now that $\rho (x_1)\not=-1$. Then,
by Proposition~\ref{pro:orbit8a}, the pair $(V,\cO )$ is not optimal
with respect to $1+c_{12}+c_{12}c_{23}$ for any $8$-orbit $\cO $.
Hence $d_8=2$ and $d_1=1$. Proposition~\ref{pro:orbit1} and $d_1=1$ imply that
$\rho (x_1)^2+\rho (x_1)+1=0$.
By Lemma~\ref{lem:Torbit8}, all Hurwitz orbits of size $8$ are conjugate. Hence
$\rdimker=2$ for all orbits $\cO$ of size $8$. 
Proposition \ref{pro:orbit8b} implies that \eqref{eq:444} holds for all $v\in V_{x_1}$,
that is, $\rho (x_2x_4)=-\rho (x_1)^{-1}$. This proves the claim.
\end{proof}

Now we discuss the Nichols algebra corresponding to $\rho $ in
Proposition~\ref{prop:b4}(2).
Assume that $\fie $ contains an element $q\in\fie $ with $q^2+q+1=0$. Recall that
$X=\mathcal{T}$. Let $\rho$ be the absolutely irreducible representation of $C_{G_X}(x_1)$ with
$\rho (x_1)=-1$, $\rho (x_4x_2)=1$.
Let $V=M(x_1,\rho)$, $a\in V_{x_1}\setminus \{0\}$, $b=q^{-1}x_3a\in V_{x_2}$,
$c=q^{-1}x_4a\in V_{x_3}$,
$d=q^{-1}x_2a\in V_{x_4}$.
The action of $G_X$ on $V$ is then determined by Table~\ref{tab:GXactionT}.

\begin{table}
\begin{equation*}
 \begin{array}{c|cccc}
   & a & b & c & d\\ \hline
   x_1 & q a & q c & qd & q b\\
   x_2 & q d & q b & -qa & -q c\\
   x_3 & q b & -q d & qc & -q a\\
   x_4 & q c & -q a & -qb & qd
 \end{array}
\end{equation*}
\caption{The action of $G_X$ on $V$, where $X=\mathcal{T}$.}
\label{tab:GXactionT}
\end{table}

\begin{prop} \label{prop:newT}
The Nichols algebra $\toba (V)$
can be presented by generators $a,b,c,d$ with defining relations
\begin{gather}
a^{3}=b^{3}=c^{3}=d^{3}=0\label{eq:T_rel1}\\
-q^2ab -q bc + ca=-q^2ac -q cd + da=0\\
q ad -q^2ba + db=q bd + q^2cb + dc=0\\
a^2bcb^2 + abcb^2a + bcb^2a^2 + cb^2a^2b + b^2a^2bc + ba^2bcb \notag \\
+ bcba^2c + cbabac + cb^2aca =0.\label{eq:T_rel6}
\end{gather}
The Hilbert series of $\toba (V)$ is
\begin{align*}
\Hilb _{\toba (V)}(t)=(6)_t^4(2)_{t^2}^2.
\end{align*}
The dimension of $\toba (V)$ is $5184$. The top degree of $\toba (V)$ is $24$.
An integral of $\toba (V)$ is given by
\[
a^2ba^2ba^2b^2a^2cb^2a^2cb^2a^2d^2.
\]
\end{prop}

\begin{proof}
The relations in
\eqref{eq:T_rel1}--\eqref{eq:T_rel6}
``generate'' a Hopf ideal of the tensor algebra $T(V)$.
Using the theory of Gr\"obner bases \cite{GBNP, GAP}, it can be seen that the quotient algebra
has the stated dimensions in each degree. Using \cite[Theorem 6.4 part (2)]{MR1994219},
it is sufficient to see that 
$a^2ba^2ba^2b^2a^2cb^2a^2cb^2a^2d^2$
does not vanish in $\mathfrak{B}(V)$
in order to prove the claim. Direct calculation gives that
\[
\partial_c\partial_c\partial_d\partial_c\partial_c\partial_d\partial_c\partial_c\partial_d\partial_d\partial_c\partial_c\partial_b\partial_b\partial_d\partial_d\partial_b\partial_a\partial_d\partial_d\partial_a\partial_a\partial_b\partial_b
\]
applied to 
$a^2ba^2ba^2b^2a^2cb^2a^2cb^2a^2d^2$
gives $-q^2$.  
This completes the proof. 
\end{proof}

\subsection{The rack $\mathcal{A}$}
\label{sec:A}

Let $X=\{1,2,3,4,5,6\}=\mathcal{A}$ and $d=\#X=6$.
Using that $X$ is braided, the rack structure of $X$ is uniquely
determined by $\varphi _1=(2\,\,3)(5\,\,6)$,
$\varphi _{2}=(1\,\,3)(4\,\,5)$.

\begin{lemma} \cite[Lemma 5.8]{Marburg}
\label{lem:b6A_centralizer}
The centralizer of $x_1$ in $G_X$ is the abelian group generated by $x_1$ and $x_4$.
These generators satisfy $x_1^2=x_4^2$.
\end{lemma}

\begin{prop}
\label{prop:b6A}
Let $\rho$ be an absolutely irreducible representation of $C_{G_X}(x_1)$ and let
$V=M(x_1,\rho)$. Then $\toba(V)$ has many cubic relations if and only if
$\rho(x_1)=-1$ and $\rho(x_4)\in\{-1,1\}$. 
\end{prop}

\begin{remark}
  The Nichols algebras $\toba (V)$ with $\rho (x_4)=-1$ and $\rho (x_4)=1$
  appeared first in \cite[Example 6.4]{MR1800714} and \cite[Def.\,2.1]{MR1667680}, respectively.
  These two Nichols algebras are twist equivalent, see \cite{twisting}.
  Their Hilbert series are given in Table~\ref{tab:nichols}.
\end{remark}

\begin{proof}
If $\rho (x_1)=-1$, then $\rho (x_4)^2=\rho (x_1)^2=1$ and hence $\rho (x_4)\in \{-1,1\}$.
Then $\toba (V)$ has many cubic relations by Theorem~\ref{thm:main}(4)$\Rightarrow $(3)
and Table~\ref{tab:nichols}.

Assume that $\toba(V)$ has many cubic relations. 
By Lemma~\ref{lem:b6A_centralizer}, the group $C_{G_X}(x_1)$ is abelian. 
Hence the degree of $\rho$ is $e=1$. 
Since $d=6$, $k_3=4$ and $m=0$, Proposition \ref{prop:verygi} implies that
\begin{equation}
\label{eq:b6A}
24 d_1+48d_8\ge 136.
\end{equation}
If $q\not=-1$, then we may set $d_8<3$ by Proposition~\ref{pro:orbit8}. This is a contradiction
to \eqref{eq:b6A}. Hence $\rho (x_1)=-1$ and the claim of the proposition follows.
\end{proof}

\subsection{The rack $\mathcal{B}$}
\label{sec:B}

Let $X=\{1,2,\dots,6\}=\mathcal{B}$ and $d=\#X=6$.
Using that $X$ is braided, the rack structure of $X$ is uniquely determined
by $\varphi _1=(2\,\,3\,\,4\,\,5)$,
$\varphi _2=(1\,\,5\,\,6\,\,3)$.

\begin{lemma} \cite[Lemma 5.10]{Marburg}
\label{lem:b6B_centralizer}
The centralizer of $x_1$ in $G_X$ is the abelian group generated by $x_1$ and $x_6$.
These generators satisfy $x_1^4=x_6^4$.
\end{lemma}

\begin{lemma}
\label{lem:Borbit8}
Let $x,y,x',y'\in X$ with $x\trid y\ne y$ and $x'\trid y'\ne y'$. Then
$\cO(x,x,y)$ and $\cO(x',x',y')$ are conjugate.
\end{lemma}

\begin{proof}
	By applying $\varphi_1$ we conclude that $\cO(1,1,2)$ is conjugate to
	$\cO(1,1,z)$ for all $z\in\{2,3,4,5\}=\{z'\in X\mid 1\trid z'\ne z'\}$. 
	Since $X$ is indecomposable, the claim follows.
\end{proof}

\begin{prop}
\label{prop:b6B}
Let $\rho$ be an absolutely irreducible representation of $C_{G_X}(x_1)$ and let
$V=M(x_1,\rho)$. Then $\toba(V)$ has many cubic relations if and only if
$\rho(x_1)=\rho(x_6)=-1$.
\end{prop}

\begin{remark}
The Nichols algebras of Prop.~\ref{prop:b6B} appeared first in \cite[Thm.\,6.12]{MR1994219}
over the complex numbers and in \cite[Prop.\,5.11]{Marburg} over arbitrary fields.
The Hilbert series of $\toba (V)$ is given in Table~\ref{tab:nichols}.
\end{remark}

\begin{proof}
If $\rho (x_1)=\rho (x_6)=-1$, then $\toba (V)$ has many cubic relations
by Theorem~\ref{thm:main}(4)$\Rightarrow $(3)
and Table~\ref{tab:nichols}.

Assume that $\toba(V)$ has many cubic relations. 
By Lemma~\ref{lem:b6B_centralizer}, the group $C_{G_X}(x_1)$ is abelian. 
Hence the degree of $\rho$ is $e=1$. Let $d_1,d_8$ be as in Proposition~\ref{prop:verygi}.
Since $d=6$, $k_3=4$ and $m=0$, Proposition \ref{prop:verygi} implies that
\eqref{eq:b6A} holds.
If $q\not=-1$, then we may assume that $d_8<3$
by Proposition~\ref{pro:orbit8}. This is a contradiction
to \eqref{eq:b6A}. Hence $\rho (x_1)=-1$.
Assume that $\rho(x_6)\ne-1$. Then $$(1+c_{12}^3)(v_1\otimes v_2)\ne 0$$ 
for $v_1\in V_{x_1}\setminus\{0\}$ and $v_2=x_3v_1\in V_{x_2}$. 
Indeed, we obtain that  
\begin{align*}
(1+c_{12}^3)(v_1\otimes v_2)&=v_1\otimes x_3v_1+x_2x_1x_3v_1\otimes x_3v_1\\
&=(v_1+x_6x_1^2v_1)\otimes x_3v_1=(v_1+x_6v_1)\otimes x_3v_1.
\end{align*}
Since all Hurwitz orbits of size $8$ are conjugate by Lemma~\ref{lem:Borbit8},
we again may assume that $d_8<3$ by Proposition~\ref{pro:orbit8}. This yields a contradiction to \eqref{eq:b6A}.
\end{proof}

\subsection{The rack $\mathcal{C}$}
\label{sec:C}

In order to avoid confusion,
let $X=\{x_1,x_2,\dots ,x_{10}\}=\mathcal{C}$. The
size of $X$ is $d=10$. The rack $X$ can be seen as the rack of transpositions in $\SG 5$.
We identify the elements of $X$ with transpositions as follows:
$x_1=(1\,2)$, $x_2=(2\,3)$, $x_3=(1\,3)$, $x_4=(2\,4)$, $x_5=(1\,4)$, $x_6=(2\,5)$, $x_7=(1\,5)$,
$x_8=(3\,4)$, $x_9=(3\,5)$, $x_{10}=(4\,5)$.

\begin{lemma} \cite[Lemma 5.8]{Marburg}
\label{lem:bC_centralizer}
The centralizer of $x_1$ in $G_X$ is the non-abelian subgroup generated by $x_1,x_8,x_9$.
These generators satisfy $x_1^2=x_8^2=x_9^2$, $x_2x_8=x_8x_2$, $x_2x_9=x_9x_2$, $x_8x_9x_8=x_9x_8x_9$.
\end{lemma}

\begin{prop}
\label{prop:bC}
Let $\rho$ be an absolutely irreducible representation of $C_{G_X}(x_1)$ and let
$V=M(x_1,\rho)$. Then $\toba(V)$ has many cubic relations if and only if
$\rho(x_1)=-1$ and $\rho(x_8)=\rho(x_9)=\pm1$. 
\end{prop}

\begin{remark}
The Nichols algebras of Proposition~\ref{prop:bC} appeared first in \cite{MR1667680} 
for $\rho(x_8)=1$ and in \cite{zoo} for $\rho(x_8)=-1$.
These two Nichols algebras are twist equivalent, see \cite{twisting}.
Their Hilbert series are given in Table~\ref{tab:nichols}.
\end{remark}

\begin{proof}
If $\rho (x_1)=-1$ and $\rho (x_8)=\rho (x_9)=\pm 1$,
then $\toba (V)$ has many cubic relations
by Theorem~\ref{thm:main}(4)$\Rightarrow $(3) and Table~\ref{tab:nichols}.

Assume that $\toba(V)$ has many cubic relations. Since $k_3=6$, the argument at the beginning
of Section~\ref{sec:proof} yields that $e=1$.
Let $d_1,d_8$ be as in Proposition~\ref{prop:verygi}.
Since $d=10$, $k_3=6$ and $m=0$, Proposition \ref{prop:verygi} implies that
\begin{align}
\label{eq:b10}
24d_1+72d_8\ge 216.
\end{align}
If $q\not=-1$, then we may assume that $d_8<3$
by Proposition~\ref{pro:orbit8}. This is a contradiction
to \eqref{eq:b10}. Hence $\rho (x_1)=-1$.
Since $x_1^2=x_8^2=x_9^2$ and $x_8x_9x_8=x_9x_8x_9$ by Lemma~\ref{lem:bC_centralizer},
we conclude that $\rho (x_8)=\rho (x_9)=\pm 1$.
\end{proof}

\subsection{The racks $\mathrm{Aff}(7,3)$ and $\mathrm{Aff}(7,5)$}
\label{sec:aff7}
Let $X=\mathrm{Aff}(7,3)$ or $X=\mathrm{Aff}(7,5)$ with $X=\{1,2,\dots ,7\}$ and let $d=\#X=7$.

\begin{prop}
\label{prop:b7}
Let $\rho$ be an absolutely irreducible representation of $C_{G_X}(x_1)$ and let
$V=M(x_1,\rho)$. Then $\toba(V)$ has many cubic relations if and only if
$\rho(x_1)=-1$. 
\end{prop}

\begin{remark}
The Nichols algebras with many cubic relations in Proposition~\ref{prop:b7} appeared first
in \cite{zoo} over $\C$ and over arbitrary fields in \cite[Prop.\,5.15]{Marburg}.
The Hilbert series of $\toba (V)$ is given in Table~\ref{tab:nichols}.
\end{remark}

\begin{proof}
If $\rho (x_1)=-1$ and $\rho (x_8)=\rho (x_9)=\pm 1$,
then $\toba (V)$ has many cubic relations
by Theorem~\ref{thm:main}(4)$\Rightarrow $(3) and Table~\ref{tab:nichols}.

Assume that $\toba(V)$ has many cubic relations. 
By \cite[Lemma 5.14]{Marburg}, the group $C_{G_X}(x_1)$ is cyclic and it is generated by $x_1$.
Hence the degree of $\rho$ is $e=1$. 
Let $d_1,d_8$ be as in Proposition~\ref{prop:verygi}.
Since $d=7$, $k_3=6$ and $m=0$, Proposition~\ref{prop:verygi} implies that \eqref{eq:b10} holds.
If $q\not=-1$, then we may assume that $d_8<3$
by Proposition~\ref{pro:orbit8}. This is a contradiction
to \eqref{eq:b10}. Hence $\rho (x_1)=-1$.
\end{proof}

\newpage

\section{Appendix A. Braided racks and Nichols algebras}
\label{sec:tables}

Tables~\ref{tab:nichols}, \ref{tab:candc} and \ref{tab:braided_racks}
contain data of finite-dimensional Nichols algebras over
groups which have a non-trivial indecomposable braided rack as support.

\begin{table}[H]
\begin{center}
\begin{tabular}{|r|r|r|l|l|}
\hline 
Rack & Rank & Dimension & Hilbert series & Remark \tabularnewline
\hline 
$\mathbb{D}_3$ & 3 & $12$ & $(2)^2_t (3)_t$ & \S\ref{sec:D3}\\
\hline
$\mathbb{D}_3$ & 3 & $432$ & $(3)_t(4)_t(6)_t(6)_{t^2}$ & Prop.~\ref{prop:newD3}, $\charf=2$\\
\hline
$\mathcal{T}$ & 4 & $36$ & $(2)^2_t (3)^2_t$ & \S\ref{sec:T}, $\charf=2$\\
\hline
$\mathcal{T}$ & 4 & $72$ & $(2)^2_t (3)_t (6)_{t}$ & \S\ref{sec:T}, $\charf\ne2$ \\
\hline
$\mathcal{T}$ & 4 & $5184$ & $(6)^4_t (2)^2_{t^2}$ & Prop.\,\ref{prop:newT} \\
\hline
$\mathcal{A}$ & 6 & $576$ & $(2)^2_t (3)^2_t (4)^2_t$ & \S\ref{sec:A} \\
\hline
$\mathcal{B}$ & 6 & $576$ & $(2)^2_t (3)^2_t (4)^2_t$ & \S\ref{sec:B} \\
\hline
$\mathrm{Aff}(7,3)$ & 7 & $326592$ & $(6)^6_{t} (7)_t$ & \S\ref{sec:aff7}\\
\hline
$\mathrm{Aff}(7,5)$ & 7 & $326592$ & $(6)^6_{t} (7)_t$ & \S\ref{sec:aff7}\\
\hline
$\mathcal{C}$ & 10 & $8294400$ & $(4)^4_t (5)^2_t (6)^4_t$ & \S\ref{sec:C}\\
\hline
\end{tabular}
\caption[aboveskip]{Finite-dimensional Nichols algebras}
\label{tab:nichols}
\end{center}
\end{table}

\begin{table}[H]
\begin{center}
\begin{tabular}{|c|c|c|}
\hline 
Rack & Generators of $C_{G_X}(x_1)$ & Linear character $\rho$ on $C_{G_X}(x_1)$ \tabularnewline
\hline 
$\mathbb{D}_3$ & $x_1$ & $\rho(x_1)=-1$\\
\hline 
$\mathbb{D}_3$ & $x_1$ & $\charf =2$, $\rho(x_1)^2+\rho (x_1)+1=0$\\
\hline 
$\mathcal{T}$ & $x_1,\;x_4x_2$ & $\rho(x_1)=-1,\;\rho(x_4x_2)=1$\\
\hline 
$\mathcal{T}$ & $x_1,\;x_4x_2$ & $\rho(x_1)^2+\rho (x_1)+1=0,\;\rho(x_4x_2x_1)=-1$\\
\hline 
$\mathcal{A}$ & $x_1,\;x_4$& $\rho(x_1)=-1,\;\rho(x_4)=\pm1$\\
\hline 
$\mathcal{B}$ & $x_1,\;x_6$ & $\rho(x_1)=\rho(x_6)=-1$\\
\hline
$\mathrm{Aff}(7,3)$ & $x_1$ & $\rho(x_1)=-1$\\
\hline 
$\mathrm{Aff}(7,5)$ & $x_1$ & $\rho(x_1)=-1$\\
\hline 
$\mathcal{C}$ & $x_1,\;x_8,\;x_9$ & $\rho(x_1)=-1,\;\rho(x_8)=\rho(x_9)=\pm1$\\
\hline
\end{tabular}
\caption[aboveskip]{Centralizers and characters}
\label{tab:candc}
\end{center}
\end{table}

\begin{table}[H]
\begin{center}
\begin{tabular}{|c|c|c|c|c|c|}
\hline
Rack & deg & size & $k_3$ & $m$ & Reference \\
\hline
$\mathbb{D}_3$ & $2$ & $3$ & $2$ & $0$ & Example \ref{example:deg2_size3} \\\hline
$\mathcal{T}$ & $3$ & $4$ & $3$ & $3$ & Prop.~\ref{prop:deg3} \\\hline
$\mathcal{A}$ & $2$ & $6$ & $4$ & $0$ & Example \ref{example:deg2_size6} \\\hline
$\mathcal{B}$ & $4$ & $6$ & $4$ & $0$ & Prop.~\ref{prop:deg4} \\\hline
$\mathcal{C}$ & $2$ & $10$ & $6$ & $0$ & Example \ref{example:deg2_size10} \\\hline
$\mathrm{Aff}(7,3)$ & $6$ & $7$ & $6$ & $0$ & Prop.~\ref{prop:deg6} \\\hline
$\mathrm{Aff}(7,5)$ & $6$ & $7$ & $6$ & $0$ & Prop.~\ref{prop:deg6} \\\hline
\end{tabular}
\caption[aboveskip]{Indecomposable braided racks occuring with Nichols algebras
with many cubic relations.}
\label{tab:braided_racks}
\end{center}
\end{table}


\section{Appendix B. Hurwitz orbits of braided racks}
\label{sec:Ho}

With Figures~\ref{fig:H3}--\ref{fig:H24} we present the isomorphism classes
of nontrivial Hurwitz orbits of braided racks. There are
nontrivial Hurwitz orbits of size 3, 6, 8, 9, 12, 16 and 24.

\begin{figure}[H]
\begin{center}
\pspicture(0,0)(0,0.7)
\psmatrix[colsep=1, mnode=circle]
A&B&C
\nccircle[linestyle=dotted, nodesep=4pt]{->}{1,1}{.4cm}
\ncline{<->}{1,1}{1,2}
\ncline[linestyle=dotted]{<->}{1,2}{1,3}
\nccircle[nodesep=4pt]{->}{1,3}{.4cm}
\endpsmatrix
\endpspicture
\caption{The Hurwitz orbit of size $3$}
\label{fig:H3}
\end{center}
\end{figure}

\begin{figure}[H]
\begin{center}
\pspicture(0,0)(0,0.7)
\psmatrix[colsep=1, mnode=circle]
A&B&C&D&E&F
\ncline{<->}{1,1}{1,2}
\ncline[linestyle=dotted]{<->}{1,2}{1,3}
\ncline{<->}{1,3}{1,4}
\ncline[linestyle=dotted]{<->}{1,4}{1,5}
\ncline{<->}{1,5}{1,6}
\ncarc[arcangle=-20, linestyle=dotted]{<->}{1,6}{1,1}
\endpsmatrix
\endpspicture
\caption{The Hurwitz orbit of size $6$}
\label{fig:H6}
\end{center}
\end{figure}

\begin{figure}[H]
\begin{center}
\hfill
\begin{minipage}{5cm}
\pspicture(0,-.5)(0,3.5)
\psmatrix[colsep=1, rowsep=0.7, mnode=circle]
A&B&C\\
D& &E\\
F&G&H
\nccircle[linestyle=dotted, nodesep=4pt]{->}{1,1}{.4cm}
\nccircle[nodesep=4pt]{->}{1,3}{.4cm}
\nccircle[nodesep=4pt]{->}{3,1}{-.4cm}
\nccircle[linestyle=dotted, nodesep=4pt]{->}{3,3}{-.4cm}
\ncline{->}{1,1}{1,2}
\ncline{->}{1,2}{2,1}
\ncline{->}{2,1}{1,1}
\ncline[linestyle=dotted]{->}{1,2}{1,3}
\ncline[linestyle=dotted]{->}{1,3}{2,3}
\ncline[linestyle=dotted]{->}{2,3}{1,2}
\ncline{->}{2,3}{3,3}
\ncline{->}{3,3}{3,2}
\ncline{->}{3,2}{2,3}
\ncline[linestyle=dotted]{->}{3,2}{3,1}
\ncline[linestyle=dotted]{->}{3,1}{2,1}
\ncline[linestyle=dotted]{->}{2,1}{3,2}
\endpsmatrix
\endpspicture
\caption{The Hurwitz orbit of size $8$}
\label{fig:H8}
\end{minipage}
\hfill
\begin{minipage}{5cm}
\pspicture(0,-.5)(0,3.5)
\psmatrix[colsep=1, rowsep=0.7, mnode=circle]
A&B&C\\
D&E&F\\
G&H&I
\ncline{->}{1,1}{1,2}
\ncline{->}{1,2}{1,3}
\ncarc[arcangle=-30]{->}{1,3}{1,1}
\ncline[linestyle=dotted]{<->}{1,1}{2,1}
\ncline[linestyle=dotted]{<->}{1,2}{2,2}
\ncline[linestyle=dotted]{<->}{1,3}{2,3}
\ncline{<->}{2,1}{3,1}
\ncline{<->}{2,2}{3,2}
\ncline{<->}{2,3}{3,3}
\ncline[linestyle=dotted]{->}{3,1}{3,2}
\ncline[linestyle=dotted]{->}{3,2}{3,3}
\ncarc[arcangle=30, linestyle=dotted]{->}{3,3}{3,1}
\endpsmatrix
\endpspicture
\caption{The Hurwitz orbit of size $9$}
\label{fig:H9}
\end{minipage}
\hfill
\end{center}
\end{figure}

\begin{figure}[H]
\begin{center}
\pspicture(0,0)(0,5)
\psmatrix[colsep=0.5, rowsep=0.3, mnode=circle]
A&&B&&C&&D\\
&E&&F&&G\\
&&H&&I\\
&&J&&K\\
&&&L
\ncline[linestyle=dotted]{->}{1,3}{1,1}
\ncline[linestyle=dotted]{->}{1,1}{2,2}
\ncline[linestyle=dotted]{->}{2,2}{1,3}
\ncline{->}{1,3}{1,5}
\ncline{->}{1,5}{2,4}
\ncline{->}{2,4}{1,3}
\ncline[linestyle=dotted]{->}{1,7}{1,5}
\ncline[linestyle=dotted]{->}{1,5}{2,6}
\ncline[linestyle=dotted]{->}{2,6}{1,7}
\ncline{->}{2,2}{3,3}
\ncline{->}{3,3}{4,3}
\ncline{->}{4,3}{2,2}
\ncline[linestyle=dotted]{->}{2,4}{3,3}
\ncline[linestyle=dotted]{->}{3,3}{3,5}
\ncline[linestyle=dotted]{->}{3,5}{2,4}
\ncline{->}{2,6}{4,5}
\ncline{->}{4,5}{3,5}
\ncline{->}{3,5}{2,6}
\ncline[linestyle=dotted]{->}{4,5}{4,3}
\ncline[linestyle=dotted]{->}{4,3}{5,4}
\ncline[linestyle=dotted]{->}{5,4}{4,5}
\ncarc[arcangle=-40]{->}{1,1}{5,4}
\ncarc[arcangle=-40]{->}{5,4}{1,7}
\ncarc[arcangle=-30]{->}{1,7}{1,1}
\endpsmatrix
\endpspicture
\caption{The Hurwitz orbit of size $12$}
\label{fig:H12}
\end{center}
\end{figure}

\begin{figure}[H]
\begin{center}
\pspicture(0,0)(0,7)
\psmatrix[colsep=0.5, rowsep=0.7, mnode=circle]
&A&&B&&C\\
D&&E&&F&&G\\
H&&&&&&I\\
J&&K&&L&&M\\
&N&&O&&P
\ncline{->}{1,2}{1,4}
\ncline{->}{1,4}{2,3}
\ncline{->}{2,3}{1,2}
\ncline{->}{1,6}{2,7}
\ncline{->}{2,7}{3,7}
\ncline{->}{3,7}{1,6}
\ncline{->}{2,1}{4,7}
\ncline{->}{4,7}{5,6}
\ncline{->}{5,6}{2,1}
\ncline{<->}{3,1}{4,1}
\ncline{->}{5,2}{5,4}
\ncline{->}{5,4}{4,3}
\ncline{->}{4,3}{5,2}
\ncline{<->}{4,5}{2,5}
\ncline[linestyle=dotted]{->}{1,4}{1,6}
\ncline[linestyle=dotted]{->}{1,6}{2,5}
\ncline[linestyle=dotted]{->}{2,5}{1,4}
\ncline[linestyle=dotted]{->}{1,2}{3,1}
\ncline[linestyle=dotted]{->}{3,1}{2,1}
\ncline[linestyle=dotted]{->}{2,1}{1,2}
\ncline[linestyle=dotted]{<->}{2,3}{4,3}
\ncline[linestyle=dotted]{<->}{3,7}{4,7}
\ncline[linestyle=dotted]{->}{2,7}{5,2}
\ncline[linestyle=dotted]{->}{5,2}{4,1}
\ncline[linestyle=dotted]{->}{4,1}{2,7}
\ncline[linestyle=dotted]{->}{4,5}{5,4}
\ncline[linestyle=dotted]{->}{5,4}{5,6}
\ncline[linestyle=dotted]{->}{5,6}{4,5}
\endpsmatrix
\endpspicture
\caption{The Hurwitz orbit of size $16$}
\label{fig:H16}
\end{center}
\end{figure}

\begin{figure}[H]
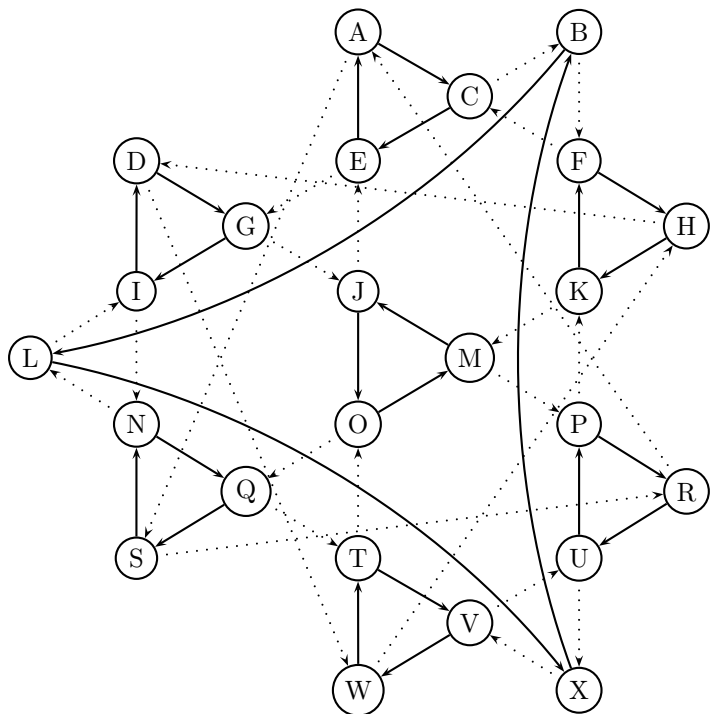

\begin{center}
\pspicture(0,0)(0,9)
\psmatrix[colsep=0.8, rowsep=0.2, mnode=circle]
&&&A&&B\\
&&&&C\\
&D&&E&&F\\
&&G&&&&H\\
&I&&J&&K\\
L&&&&M\\
&N&&O&&P\\
&&Q&&&&R\\
&S&&T&&U\\
&&&&V\\
&&&W&&X
\striangle142534
\striangle324352
\striangle364756
\striangle547465
\striangle728392
\striangle768796
\striangle94{10}5{11}4
\ttriangle163625
\ttriangle344354
\ttriangle527261
\ttriangle566576
\ttriangle748394
\ttriangle96{11}6{10}5
\ncline[linestyle=dotted]{->}{1,4}{9,2}
\ncline[linestyle=dotted]{->}{9,2}{8,7}
\ncline[linestyle=dotted]{->}{8,7}{1,4}
\ncline[linestyle=dotted]{->}{3,2}{11,4}
\ncline[linestyle=dotted]{->}{11,4}{4,7}
\ncline[linestyle=dotted]{->}{4,7}{3,2}
\ncarc[arcangle=20]{->}{1,6}{6,1}
\ncarc[arcangle=20]{->}{6,1}{11,6}
\ncarc[arcangle=20]{->}{11,6}{1,6}
\endpsmatrix
\endpspicture
\caption{The Hurwitz orbit of size $24$}
\label{fig:H24}
\end{center}
\end{figure}


\clearpage

\textbf{Acknowledgement.}
We are grateful to Gunter Malle for introducing us $3$-transpo\-sition
groups. Many thanks are going to Volkmar Welker for providing us with the
references to percolation theory. Finally, we thank the referees
for their careful reading and for their hints leading to improvements of the paper.

I.\,H. was supported by the German Research Foundation via a Heisenberg
fellowship.  L.\,V. was supported by CONICET and DAAD. L.\,V. thanks
Philipps-Universit\"at Marburg for the support of his visit from
December 2010 to April 2011.

\bigskip

\newcommand{\etalchar}[1]{$^{#1}$}
\def\cprime{$'$}
\providecommand{\bysame}{\leavevmode\hbox to3em{\hrulefill}\thinspace}
\providecommand{\MR}{\relax\ifhmode\unskip\space\fi MR }
\providecommand{\MRhref}[2]{%
  \href{http://www.ams.org/mathscinet-getitem?mr=#1}{#2}
}
\providecommand{\href}[2]{#2}

\end{document}